\documentclass[12pt]{article}
\usepackage[centertags]{amsmath}
\usepackage{amsfonts}
\usepackage{amssymb}
\usepackage{amsthm}
\usepackage{amsmath,mathrsfs}
\usepackage{amsmath,amscd}
\usepackage[matrix,arrow,curve,cmtip]{xy}
\title{\textbf{Internal Bousfield Localizations}}
\author{Renaud Gauthier \footnote{rg.mathematics@gmail.com} \\ \\}
\theoremstyle{definition}

\newtheorem{iRhfC}{Definition}[subsubsection]
\newtheorem{iLhfC}[iRhfC]{Definition}
\newtheorem{equ}{Proposition}[subsection]
\newtheorem{17.7.7}[equ]{Theorem[17.7.7 mod]}
\newtheorem{Yoneda}{Theorem}[subsection]
\newtheorem{Fibrant}{Proposition}[subsection]
\newtheorem{higheradj}[Fibrant]{Theorem}
\newtheorem{Fibrant2}[Fibrant]{Proposition}
\newtheorem{HomAXsimpres}{Proposition}[subsection]
\newtheorem{17.2.5}{Definition}[subsection]
\newtheorem{17.2.7}[17.2.5]{Definition}
\newtheorem{17.2.11}[17.2.5]{Theorem}
\newtheorem{17.2.10}[17.2.5]{Proposition}
\newtheorem{3.1.5}{Proposition[3.1.5 mod]}[subsection]
\newtheorem{4.2.3}[3.1.5]{Proposition[4.2.3 mod]}
\newtheorem{4.2.5}[3.1.5]{Proposition[4.2.5 mod]}
\newtheorem{12.5.2}{Proposition[12.5.2 mod]}[subsection]
\newtheorem{3.2.3}{Proposition[3.2.3 mod]}[section]
\newtheorem{gloc implies loc}[3.2.3]{Lemma}
\newtheorem{lequ implies glequ}[3.2.3]{Lemma}
\newtheorem{ideal}[3.2.3]{Definition}
\newtheorem{iGammalocequideal}[3.2.3]{Lemma}
\newtheorem{Gammalocequideal}[3.2.3]{Lemma}
\newtheorem{ideal2}[3.2.3]{Definition}
\newtheorem{iKcolobjideal}[3.2.3]{Lemma}
\newtheorem{diagHom}[3.2.3]{Lemma}
\newtheorem{4.5.6}[3.2.3]{Proposition[4.5.6 mod]}
\newtheorem{4.5.2}[3.2.3]{Lemma[4.5.2 mod]}
\newtheorem{4.5.1}[3.2.3]{Proposition[4.5.1 mod]}
\newtheorem{Glocal}{Definition}[section]
\newtheorem{Glocalequ}[Glocal]{Definition}
\newtheorem{4.1.1}[Glocal]{Theorem}
\newtheorem{3.2.4}{Proposition[3.2.4 mod]}[section]
\newtheorem{equ implies K-colequ}[3.2.4]{Proposition}
\newtheorem{5.3.2}[3.2.4]{Lemma [5.3.2 mod]}
\newtheorem{gK-col implies K-col}[3.2.4]{Lemma}
\newtheorem{K-colequ implies gK-colequ}[3.2.4]{Proposition}
\newtheorem{colocal}{Definition}[section]
\newtheorem{colequ}[colocal]{Definition}
\newtheorem{Kcolequ}[colocal]{Definition}
\newtheorem{RBous}[colocal]{Theorem[5.1.1 mod]}

\DeclareMathOperator*{\adj}{\rlarr}
\newcommand{\beq}{\begin{equation}}
\newcommand{\eeq}{\end{equation}}
\newcommand{\rarr}{\rightarrow}
\newcommand{\Lrarr}{\longrightarrow}
\newcommand{\rarrequ}{\stackrel{\simeq}{\Lrarr}}
\newcommand{\rlarr}{\rightleftarrows}

\newcommand{\cC}{\mathcal{C}}
\newcommand{\cCop}{\cC^{\op}}
\newcommand{\cD}{\mathcal{D}}
\newcommand{\cE}{\mathcal{E}}

\newcommand{\cK}{\mathcal{K}}
\newcommand{\cM}{\mathcal{M}}
\newcommand{\cN}{\mathcal{N}}

\newcommand{\cV}{\mathcal{V}}

\newcommand{\bL}{\mathbb{L}}
\newcommand{\bR}{\mathbb{R}}
\newcommand{\bU}{\mathbb{U}}

\newcommand{\diag}{\text{diag}}
\newcommand{\Hom}{\text{Hom}}

\newcommand{\map}{\text{map}}

\newcommand{\op}{\text{op}}
\newcommand{\oT}{\otimes}

\newcommand{\Set}{\text{Set}}
\newcommand{\SetD}{\Set_{\Delta}}

\newcommand{\uHom}{\underline{\Hom}}

\newcommand{\cMDop}{(\cM)^{\Delta^{\op}}}
\newcommand{\DCop}{\cD_0^{\cCop}}
\newcommand{\DiCop}{\cD_i^{\cCop}}
\newcommand{\DipCop}{\cD_{i+1}^{\cCop}}
\newcommand{\DCopw}{\cD_0^{\cCop \; \wedge}}
\newcommand{\DiCopw}{\cD_{i}^{\cCop \; \wedge}}

\newcommand{\DCopDop}{(\DCop)^{\Delta^{\op}}}

\newcommand{\DCopwDop}{(\DCopw)^{\Delta^{\op}}}

\newcommand{\DCopwD}{(\DCopw)^{\Delta}}

\newcommand{\hiCop}{h_i^{\cCop}}
\newcommand{\kipCop}{k_{i+1}^{\cCop}}

\newcommand{\sPr}{\text{sPr}}
\newcommand{\ssPr}{\text{ssPr}}

\newcommand{\uHomMI}{\uHom_{\cM^I}}

\newcommand{\HomM}{\Hom_{\cM}}
\newcommand{\uHomM}{\uHom_{\cM}}
\newcommand{\HomD}{\Hom_{\cD_0}}

\newcommand{\Icof}{I\text{-Cof}}
\newcommand{\Iinj}{I\text{-inj}}
\newcommand{\uHomD}{\uHom_{\cD_0}}

\newcommand{\HomDCop}{\Hom_{\DCop}}

\newcommand{\uHomDCop}{\uHom_{\DCop}}

\newcommand{\HomDCopw}{\Hom_{\DCopw}}

\newcommand{\uHomDCopw}{\uHom_{\DCopw}}

\newcommand{\LGDCop}{L_{\Gamma}\DCop}
\newcommand{\RKLGDCop}{R_{\cK_0}\LGDCop}
\newcommand{\RKM}{R_{\cK}\cM}
\newcommand{\wLG}{\widetilde{\Lambda \Gamma}}

\begin{document}
\maketitle
\begin{abstract}
We develop the notion of left and right Bousfield localizations in proper, cellular symmetric monoidal model categories with cofibrant unit, using homotopy function complexes defined by internal Hom objects instead of Hom sets.
\end{abstract}

\newpage
\section{Introduction}
From \cite{TV}, recall that one can define a stack in $\sPr(T)$, the model category of simplicial presheaves over a simplicial model category $T$, as an object $F \in \sPr(T)$ satisfying hyperdescent, meaning being local with respect to a certain class of hypercovers, or being local with respect to local equivalences, both concepts necessitating that we define a topology on $T$. That this is really one localization is due to the fact that local equivalences (or $\pi_*$-equivalences are they are sometimes called, see \cite{TV}), maps $F \rarr G$ in $\sPr(T)$ such that for all $n > 0$, we have an isomorphism of sheaves $\pi_n(F) \rarr \pi_n(G)$, among other things, are also hypercovers $F \rarr G$ in $\ssPr(T)$ (viewed as constant simplicial objects), maps such that for all $n \geq 0$:
\beq
F^{\mathbb{R} \Delta^n} \rarr F^{\mathbb{R}\partial \Delta^n} \times^h_{G^{\mathbb{R}\partial \Delta^n}}G^{\mathbb{R} \Delta^n} \nonumber
\eeq
is a $\tau$-covering, $\tau$ being a topology on $T$. Such a map $F \rarr G$ would be a $\mathbb{S}$-colocal object in the language of \cite{Hi}. Suppose we consider objects $\tilde{k_0}$ other than the sphere spectrum $\mathbb{S}$ in this definition of a local equivalence, cosimplicial resolutions of some $k_0$, object of some class $K_0$. The idea is that we want to consider reference objects that are very general, in keeping with the philosophy of relative Algebraic Geometry. Suppose further we consider functors not valued in $\SetD$, the category of simplicial sets, but in some category $\cD$, which for the moment we can suppose is an internal, proper, cellular model category. It would be natural then for the sake of localization to use internal Hom objects for the definition of homotopy function complexes, as opposed to using Hom sets. As a matter of fact, we will closely follow Hirschhorn's work (\cite{Hi}) regarding Bousfield localizations and make the appropriate changes.\\

We start with a $\bU$-small pseudo-model category $(\cC,W)$ (\cite{TV}), $\cD_0$ a proper, cellular, symmetric monoidal model category, and we consider the functor category $\DCop$. Prestacks in this setting are functors $F: \cCop \rarr \cD_0$ mapping equivalences to equivalences; if $x \rarr y$ is in $W$, then $Fx \rarr Fy$ is an equivalence in $\cD_0$. If $\cD_0 = \SetD$, one can use the classical Yoneda lemma: $Fx \simeq \Hom(h_x, F)$, $h_x = \Hom_{\cC}(-,x)$, $\cC$ a simplicial category, to see equivalence preserving functors as local objects. For $\cD_0$-enriched functors, we must consider enriched Yoneda: $Fx \simeq \int_{y \in \cCop} \uHom(h_x(y), Fy)$. Then $Fx \rarrequ Fy$ will follow from defining a notion of local object using internal Hom objects. This leads us to defining the internal left Bousfield localization $\LGDCop = \DCopw$ of $\DCop$ with respect to $\Gamma = \{h_y \rarr h_x \; | \; x \rarr y \in W^{\op} \}$. However, as will become clear later, for the sake of generalizing classical results from \cite{Hi} to the setting where Hom sets are replaced by internal Hom objects, we denote by $\Gamma^0$ the collection we just defined, and we will take a localization of $\DCop$ with respect to:
\beq
\Gamma := \Gamma^0 \oT \DCop = \{ h_y \oT F \rarr h_x \oT F \, | \, x \rarr y \in W^{\op},\, F \in \DCop \} \nonumber  
\eeq

Our philosophy at this point departs markedly from the standard philosophy of Homotopical Algebraic Geometry (\cite{TV}) in that we have no topology on $\cC$, and we limit ourselves to considering only one condition defining $\pi_*$-local equivalences, namely $\pi_n(F) \rarrequ \pi_n(G)$ for all $n>0$, which we internalize, using not spheres, but arbitrary objects $k_0$ of some class of objects $K_0$ of $\DCop$:
\beq
\uHom(\tilde{k_0}, \hat{F}) \rarrequ \uHom(\tilde{k_0}, \hat{G}) \label{first}
\eeq
where using the sphere spectrum $\mathbb{S}$ for $\pi_*$-local equivalences is replaced by using a cosimplicial resolution $\tilde{k_0}$ of a single object $k_0$. Here $\hat{F}$ is a fibrant approximation to $F$. \eqref{first} would define what it means for $F \rarr G$ to be an internal $K_0$-colocal equivalence. From there we are naturally led to considering an internal right Bousfield localization $\RKLGDCop$ of $\DCopw$ with respect to $\cK_0$, the class of internal $K_0$-colocal equivalences.\\

In Section 2 we present the main construction, giving a category of prestacks $\DCopw = \LGDCop$ from localizing $\DCop$ along a subset $\Gamma$, followed by further taking a right Bousfield localization along a class of colocal equivalences, all concepts being redefined using internal Homs. In Section 3 we present foundational results, such as the enriched Yoneda lemma, and localization using internal Homs for homotopy function complexes. In Section 4 we discuss cardinality arguments needed in the proof of our main result, the existence of a left Bousfield localization using internal Homs instead of Hom sets. In Section 5 we present technical results needed to prove the existence of such a localization, which is itself given in Section 6. In Section 7 we give those results needed to prove the existence of a right Bousfield localization using internal Homs, which is stated and proved in Section 8.\\

\newpage

Relation to past work: it was pointed out to the author by J. Gutierrez and D. White, that the present work is quite close to past work on the subject. In particular they both referenced two papers the author was wholly unaware of, namely \cite{B} and \cite{GR}. As a matter of fact, Barwick's work is so close to the present one, the original thought of using the Hom from a Quillen adjunction of two variables to define Bousfield localizations must be credited to him. Presently we discuss localizations of symmetric monoidal model categories, and some work has been done on the subject, albeit in a classical sense, not using internal Homs. This work can be found in \cite{W} and \cite{WY} where such localizations are referred to as monoidal Bousfield localizations. \\

To come back to \cite{W} and \cite{GR}, there are slight differences to be noted. Barwick works with tractable categories, which are combinatorial, hence cofibrantly generated. We work with cellular model categories, which are also cofibrantly generated. Barwick observes that he does not know of any left proper combinatorial model category that is not tractable. On our part, we do not see how to relate cellular model categories with tractable model categories; our objects of study seem different. Additionally, Barwick works with $\cV$-enriched categories $\cC$, for $\cV$ a symmetric monoidal model category, and the ``internal" Homs he uses, derived from a Quillen adjunction of two variables due to this enrichment, are objects in $\cV$. In particular he shows that for a small site $\cC$, for $\cV$ a tractable symmetric monoidal model category with cofibrant unit, $\cV^{\cC}$ is a $\cV$-model category, so is enriched in $\cV$, on which we define an injective local model structure as an enriched left Bousfield localization of $\cV^{\cC}$ with its injective model structure. What we do instead is show that $\cV^{\cC}$ is a symmetric monoidal model category with an internal Hom, and it is those internal Homs we use in the definition of our version of Bousfield localizations. For this reason we call them internal Bousfield localizations, as opposed to enriched Bousfield localizations, which use Hom objects valued in $\cV$.\\

Regarding the Bousfield localization itself, Barwick's construction is an astute one. He does not define a new Bousfield localization. He uses what we refer to as the classical Bousfield localization, that of Hirschhorn in \cite{Hi}. If $H$ is a set of maps we are localizing with respect to, $S$ is a set representing homotopy classes of $H$, $I$ a generating set of cofibrations with cofibrant domains, he shows his enriched Bousfield localization of a tractable, left proper model category $\cC$ is none other than $L_{I \Box S} \cC$ in the classical sense, and then uses Thm 17.4.16 of \cite{Hi} to have his enriched Homs appear. His proof is short, precisely because of his ingenious use of that result, as opposed to ours, which is a tedious reworking of Hirschhorn's work in \cite{Hi} about left Bousfield localization, in the event that Hom sets are replaced by internal Hom objects.\\

Barwick's work is elegantly generalized in \cite{GR}, in the larger setting of combinatorial model categories. Gutierrez and Roitzhem consider a Quillen adjunction of two variables between combinatorial model categories, $\cC \times \cD \rarr \cE$, and they define the left Bousfield localization $L_S \cE$ of $\cE$ with respect to a set $S$ of morphisms in $\cC$, which they call the $S$-local model structure on $\cE$. In the notations of \cite{B} as used above, if we consider a Quillen adjunction of two variables $\cC \otimes \cV \rarr \cC$ associated with a $\cV$-enrichment, then $L_S \cC = L_{I \Box S} \cC$, which corresponds to a $\cV$-enriched left Bousfield localization of $\cC$ with respect to $S$ as defined by Barwick.\\

\section{Construction}
We fix a universe $\bU$. Let $(\cC,W)$ be a $\bU$-small pseudo-model category, $\cD_0$ a proper, cellular model category, $\DCop$ the model category of functors from $\cCop$ to $\cD_0$. It is also a proper, cellular model category (Thm 13.1.14 and Prop. 12.1.5 of \cite{Hi}). We aim to take left and right Bousfield localizations of $\DCop$ with respect to certain classes of maps, in a generalized sense. By this we mean we will use internal Hom objects in the definition of homotopy function complexes instead of Hom sets. We will prepare the ground for operating such localizations. This is what ``construction" is in reference to. We will first introduce $\Gamma^0 = \{h_y \rarr h_x \;|\; x \rarr y \in W^{\op} \}$, where $h_a = \Hom_{\cC}(-,a)$ and $W$ denotes the class of weak equivalences in $\cC$. This collection is enlarged to $\Gamma = \Gamma^0 \oT \DCop$ and we construct $\LGDCop = \DCopw$, the internal left Bousfield localization of $\DCop$ with respect to $\Gamma$. This should provide a notion of category of pre-stacks, since such a localization in particular provides a localization with respect to $\Gamma^0$. \\

\newpage
\noindent

From that point forward, we will introduce a class $K_0$ of objects in $\DCopw$ and consider the class $\cK_0$ of internal $K_0$-colocal equivalences. We will then construct $R_{\cK_0} \DCopw$ the internal right Bousfield localization of $\DCopw$ with respect to $\cK_0$, and this will constitute a first step towards defining a generalized notion of stacks on $\cC$.\\

Many of the results we will discuss in this work are stated and proved in the classical sense in \cite{Hi}. To make comparison with those original statements easier, next to each claim we will put in bracket the original indexation in \cite{Hi} along with ``mod", indicating that we are stating a modified version thereof.\\

\subsection{Preserving equivalences}
In a first time, we would like functors $F: \cCop \rarr \cD_0$ to map equivalences $u: x \rarr y$ in $\cCop$ to equivalences $Fx \rarr Fy$ in $\cD_0$. For this purpose we introduce the following class:
\beq
\Gamma^0 = \{h_y \rarr h_x \;|\; x \rarr y \in W^{\op} \} \nonumber
\eeq
If we were to use the classical Yoneda lemma, to obtain the desired equivalences, it would suffice to show we have:
\beq
\Hom(h_x, F) \rarrequ \Hom(h_y,F) \nonumber
\eeq
which would make $F$ $\Gamma^0$-local in the terminology of \cite{Hi}. However, $F$ is $\cD_0$-valued, not $\Set$-valued, so we use the enriched Yoneda lemma (\cite{K}). In order to do so we start to use internal Homs (whose existence will be warranted below), hence we would want:
\beq
\uHomDCop(h_x, F) \rarrequ \uHomDCop(h_y,F) \nonumber
\eeq
We will actually show:
\beq
\uHomDCop(\tilde{h_x}, \hat{F}) \rarrequ \uHomDCop(\tilde{h_y},\hat{F}) \nonumber
\eeq
in the Reedy structure of $\DCopDop$ ($\tilde{h}$ cofibrant approximation to $h$, $\hat{F}$ simplicial resolution of $F$) which will imply $F$ maps equivalences to equivalences. \\

\newpage

As a matter of fact, we will show the more general result:
\beq
\uHomDCop(\tilde{h_x} \oT \tilde{G}, \hat{F}) \rarrequ \uHomDCop(\tilde{h_y} \oT \tilde{G},\hat{F}) \nonumber
\eeq
for all $G \in \DCop$, which will imply the equivalence prior.\\

\subsection{Monoidal structure}
\subsubsection{Definitions}
In order to use internal Homs, we first ask that $\cD_0$ be a monoidal model category (\cite{Ho}), with internal Hom $\uHomD$. For simplicity, every notion of \cite{Hi} using Hom sets that is generalized using internal Homs will be referred to as a \textbf{generalized} \index{generalized concept} or as an \textbf{internal concept}. \index{internal concept} For instance, we would define:
\beq
\map(k_0,F) = \uHomDCopw(\tilde{k_0}, \hat{F}) \nonumber
\eeq
as an internal right homotopy function complex (abbr. iRhfC), where $\tilde{k_0}$ is a cofibrant approximation to $k_0$, and $\hat{F}$ is a simplicial resolution of $F$. The existence of an internal Hom in $\DCop$ will be discussed below.\\

For later purposes, we make the following definitions:
\begin{iRhfC}
An \textbf{internal right homotopy function complex} \index{function complex! internal right homotopy} (abbr. iRhfC) is an object of $\DCopDop$ of the form $\uHomDCop(\tilde{X}, \hat{Y})$, where $\tilde{X}$ is a cofibrant approximation to $X$ and $\hat{Y}$ is a simplicial resolution of $Y$.
\end{iRhfC}
\begin{iLhfC}
An \textbf{internal left homotopy function complex} \index{function complex! internal left homotopy} (abbr. iLhfC) is an object of $\DCopDop$ of the form $\uHomDCop(\tilde{X}, \hat{Y})$, where $\tilde{X}$ is a cosimplicial resolution of $X$ and $\hat{Y}$ is a fibrant approximation to $Y$.
\end{iLhfC}
Note that $h_x = \Hom_{\cC}(-,x)$ is a set-valued functor, not an object of $\DCop$. Hence we ask that $\cC$ be a $\cD_0$-enriched category(\cite{K}). This we can do since $\cD_0$ is a monoidal category.\\

\subsubsection{$\DCop$ is a monoidal model category}
We will also need that $\DCop$ be a monoidal model category. We define the monoidal structure on $\DCop$ point-wise: if $F,G \in \DCop$, then for any $x \in \cC$, $(F \otimes G)(x) = Fx \otimes Gx$. Endow $\DCop$ with the injective model structure; equivalences and cofibrations are defined point-wise. This makes the tensor product on $\DCop$ a Quillen bifunctor. Indeed, let $\alpha: U \rarr V$ be a cofibration in $\DCop$, $\beta: W \rarr X$ a cofibration in $\DCop$ as well, we need:
\beq
\alpha \Box \beta: (V \otimes W) \coprod_{U\otimes W} (U \otimes X) \rarr V \otimes X \nonumber
\eeq
to be a cofibration in $\DCop$, trivial if either of $\alpha$ or $\beta$ is (\cite{Ho}). Since we take the injective model structure on $\DCop$, we have to check that pointwise: let $x \in \cCop$. We are looking at:
\beq
\alpha \Box \beta(x): (Vx \otimes Wx) \coprod_{Ux \otimes Wx} (Ux \otimes Xx) \rarr Vx \otimes Xx \nonumber
\eeq
Now $U \rarr V$ cofibration in $\DCop$ with the injective model structure means $Ux \rarr Vx$ cofibration in $\cD_0$, and $W \rarr X$ cofibration means $Wx \rarr Xx$ cofibration in $\cD_0$, and it follows that the above map is a cofibration in $\cD_0$ (since it is a monoidal model category), and this for all $x \in \cCop$, hence a cofibration in $\DCop$. Since equivalences are defined pointwise, we also have that $\alpha \Box \beta$ is trivial if either of $\alpha $ or $\beta$ is. The second condition for being a monoidal model category (\cite{Ho}) is that if $Q1 \rarr 1$ is a cofibrant approximation to the unit 1, then for all $X$ cofibrant, $Q1 \otimes X \rarr 1 \otimes X$ is a weak equivalence. Here we assume the unit is cofibrant, a condition we will use later. If that is the case, this condition is automatically satisfied. This makes $\DCop$ into a monoidal model category. Observe that having $\otimes$ a Quillen bifunctor, if $C$ is cofibrant in $\DCop$, then $C \otimes -: \DCop \rarr \DCop$ is left Quillen, so preserves cofibrations and trivial cofibrations, a fact that will be very important in the work that follows.\\

\newpage

\subsubsection{Construction of $\uHomDCop$}
The internal Hom of $\DCop$ is formally defined as follows: an element $\alpha \in \HomDCop(F \otimes G, H)$ is defined pointwise: for $x \in \cC$, $\alpha(x): Fx \otimes Gx \rarr Hx$ in $\cD_0$. Now:
\beq
\HomD(Fx \otimes Gx, Hx) \cong \HomD(Fx, \uHomD(Gx, Hx)) \nonumber
\eeq
so $\alpha(x)$ would correspond to some $\beta(x) \in \HomD(Fx, \uHomD(Gx, Hx))$. Letting
\beq
\uHomD(Gx,Hx) = \uHomDCop(G,H)(x) \label{daguer}
\eeq
we have $\uHomDCop(G,H) \in \DCop$, and with this notation:
\beq
\HomDCop(F \otimes G, H) \cong \HomDCop(F, \uHomDCop(G, H)) \nonumber
\eeq
This can be formalized using the language of ends (\cite{McL}). It suffices to write, still using the identification \eqref{daguer}:
\begin{align}
\HomDCop(F \otimes G, H)&\cong \int_{x \in \cCop} \HomD(Fx \otimes Gx, Hx) \nonumber \\
&\cong \int_{x \in \cCop} \HomD(Fx, \uHomD(Gx,Hx)) \nonumber \\
&=\int_{x \in \cCop} \HomD(Fx, \uHomDCop(G,H)(x)) \nonumber \\
&\cong \HomDCop(F, \uHomDCop(G,H)) \nonumber
\end{align}

We can make this more precise. From now on, we will assume that $(\cD_0, \otimes)$ is also symmetric. Following \cite{GK}, and working in full generality for later purposes, consider $\cM$ a closed symmetric monoidal category (i.e. it has an internal Hom), which presently is $\cD_0$ for us. An $\cM$-module structure on a category $\cC$ is given by an action:
\begin{align}
\otimes: \cC \times \cM & \rarr \cC \nonumber \\
(X,M) & \mapsto X \otimes M \nonumber
\end{align}
This action is closed if we have two functors:
\begin{align}
\map_{\cC}: \cCop \times \cC &\rarr \cM \nonumber \\
(X,Y) &\mapsto \map_{\cC}(X,Y) \nonumber
\end{align}
and:
\begin{align}
\text{exp}: \cM^{\op} \times \cC &\rarr \cC \nonumber \\
(K,Y) &\mapsto Y^K \nonumber
\end{align}
in such a manner that we have, for all $K \in \cM$, $X,Y \in \cC$, the following natural isomorphisms:
\beq
\Hom_{\cC}(X \otimes K, Y) \cong \HomM(K, \map_{\cC}(X,Y)) \cong \Hom_{\cC}(X, Y^K) \nonumber
\eeq
Now consider the following functor, where $I$ is an indexing set:
\begin{align}
\cM^I \times \cM \rarr &\cM^I \nonumber \\
(X,K) \mapsto &X \otimes K \nonumber
\end{align}
defined by $(X \otimes K)_i = X_i \otimes K $. This defines an $\cM$-module structure on $\cM^I$. It is closed if we use the following definitions: for $X,Y \in \cM^I$, $K \in \cM$, define $Y^K=\uHomM(K,Y)$ by $\uHomM(K,Y)_i = \uHomM(K,Y_i)$ and:
\beq
\map_{\cM^I}(X,Y) = \int_i \uHomM(X_i,Y_i) \nonumber
\eeq
In particular for $\cM = \cD_0$ and $I = \cCop$, this gives us, for $X,Y \in \DCop$, $K \in \cD_0$:
\beq
\HomDCop(X \otimes K, Y) \cong \HomDCop(X, \uHomD(K,Y)) \cong \HomD(K, \map_{\DCop}(X,Y)) \nonumber
\eeq
Coming back to the general case, consider:
\begin{align}
h_i:I &\rarr \cM \nonumber \\
j &\mapsto h_i(j) = \coprod_{I(i,j)} 1 \nonumber
\end{align}
where 1 is the unit of $\cM$. Define a monoidal structure on $\cM^I$ as follows: for $X,Y \in \cM^I$, let $(X \otimes Y)_i = X_i \otimes Y_i$, making $(\cM^I, \otimes)$ into a symmetric monoidal category, which is furthermore closed, with internal Hom given by:
\begin{align}
\uHomMI(X,Y)_i &= \map_{\cM^I}(h_i \otimes X,Y) \nonumber \\
&=\int_{j \in I} \uHomM(h_i(j) \otimes X_j, Y_j) \nonumber
\end{align}
For $\cM = \cD_0$ and $I = \cCop$, this gives us the internal Hom in $\DCop$:
\begin{align}
\uHomDCop(F,G)(x) &= \map_{\DCop}(h_x \otimes F,G) \label{intHomDCop} \\
&=\int_{y \in \cCop}\uHomD(h_x(y) \otimes Fy , Gy) \nonumber \\
&=\int_{y \in \cCop} \uHomD(h_x(y) , \uHomD(Fy, Gy)) \label{star1} \\
	&= \map_{\DCop}(h_x, \uHomD(F-,G-)) \nonumber
\end{align}
We proceed from \eqref{star1} to get \eqref{daguer} as follows. For any $W$:
\begin{align}
	\HomD(W,\, &\uHomDCop(F,G)(x)) \nonumber \\
	&= \HomD(W, \int_{y \in \cCop} \uHomD(h_x(y) , \uHomD(Fy, Gy))) \nonumber \\
&= \int_{y \in \cCop}\HomD(W, \uHomD(h_x(y) , \uHomD(Fy, Gy))) \nonumber \\
&= \int \HomD(W \otimes h_x(y) , \uHomD(Fy,Gy)) \nonumber \\
&= \HomDCop(W \otimes h_x, \uHomD(F-,G-)) \nonumber
\end{align}
Now we use the fact that if we define the evaluation functor $Ev_i: \cM^I \rarr \cM$ by $Ev_i(X) = X_i$ and $F : \cM \rarr \cM^I $ by $F_i(M) = h_i \otimes M$, then we have $F_i  \dashv  Ev_i$, which reads, for $\cM = \cD_0$ and $I = \cCop$:
\beq
\HomDCop(h_x \otimes M, G) \cong \HomD(M, G(x)) \nonumber
\eeq
using this in the previous computation gives us:
\beq
\HomD(W, \uHomDCop(F,G)(x)) = \HomD(W, \uHomD(Fx,Gx)) \nonumber
\eeq
and this being true for all $W$, we have:
\beq
\uHomDCop(F,G)(x) \cong \uHomD(Fx,Gx) \nonumber
\eeq

\newpage

\subsubsection{More on tensor products}
As part of the construction of $\uHomDCop$, and for later purposes, we use the following result of \cite{GK}: if $I$ is a Reedy category, $\cM$ is a cofibrantly generated symmetric monoidal model category with cofibrant unit, then $\cM^I$ is a closed symmetric monoidal model category. We use this for $\cM = \DCopw$, cellular, in particular cofibrantly generated, with a symmetric monoidal structure and a cofibrant unit. It follows that if we take $I = \Delta^{\op}$, then $\DCopwDop$ is a symmetric monoidal model category. In particular $\otimes$ is a Quillen bifunctor.\\

Note that one has to keep track of what tensor products are in use, whether they are part of a monoidal structure, or a generalization thereof, for example when we have model structures. For instance:
\begin{align}
\otimes : \DCopwD \times \DCopw & \rarr \DCopwD \nonumber \\
(\tilde{k_0},\tilde{D}) & \mapsto \tilde{k_0} \otimes \tilde{D} \nonumber
\end{align}
defines a $\DCopw$-module structure on $\DCopwD$. If $\tilde{D}$ is a cofibrant approximation to $D$, $\tilde{k_0}$ a cosimplicial resolution of $k_0$, in writing:
\beq
\HomDCopw(\tilde{D}, \uHomDCopw(\tilde{k_0}, \hat{X})) \cong \HomDCopw(\tilde{D} \otimes \tilde{k_0}, \hat{X}) \nonumber
\eeq
as provided by Thm \ref{3.4.2}, the tensor product $\tilde{D} \otimes \tilde{k_0}$ makes sense. Here we have used $\uHomDCopw = \uHomDCop$, for the simple reason that the definition of an internal Hom is peculiar to the monoidal structure, not the model structure, so is preserved after localization. This observation, which we will make more precise, presupposes that $\DCopw$ is a monoidal model category, which we now show.

\subsubsection{$\DCopw$ is a monoidal model category}
We show $\DCopw$ is a monoidal model category. We use the same tensor product as for $\DCop$, and we need that it be a Quillen bifunctor, that is it satisfies the pushout-product axiom. First cofibrations in $\LGDCop = \DCopw$ are those of $\DCop$ as well. It being a monoidal model category, it satisfies the pushout-product axiom, so if $f: U \rarr V$ and $g: X \rarr Y$ are cofibrations in $\DCopw$, so is $f \Box g$. If in addition either of $f$ or $g$ is an i$\Gamma$-local equivalence (Def. \ref{Glocalequ}), so is $f \Box g$: without loss of generality, let's assume $f$ is a trivial cofibration in $\DCopw$. Now i$\Gamma$-local equivalences form an ideal class as shown in Section 5, which implies $U \otimes X \rarr V \otimes X$ is an i$\Gamma$-local equivalence as well, hence so is $\tilde{U} \otimes \tilde{X} \rarr \tilde{V} \otimes \tilde{X}$. $\otimes$ being a Quillen bifunctor on $\DCop$, for $\tilde{X}$ a cofibrant approximation to some object $X$ of $\DCop$, $\tilde{U} \otimes \tilde{X} \rarr \tilde{V} \otimes \tilde{X}$ is a cofibration in $\DCop$ if we take $\tilde{U} \rarr \tilde{V}$ a fibrant cofibrant approximation to $U \rarr V$ that is a cofibration in $\DCop$. Hence $\tilde{U} \otimes \tilde{X} \rarr \tilde{V} \otimes \tilde{X}$ is an i$\Gamma$-local equivalence and a cofibration, i.e. a trivial cofibration. Now:
\beq
\xymatrix{
\tilde{U} \otimes \tilde{X} \ar[d] \ar[r]^{triv.cof.} & \tilde{V} \otimes \tilde{X} \ar[d] \nonumber \\
\tilde{U} \otimes \tilde{Y} \ar[dr]_{triv.cof.} \ar[r]^-{triv.cof.} &\tilde{U} \otimes \tilde{Y} \coprod_{\tilde{U} \otimes \tilde{X}} \tilde{V} \otimes \tilde{X} \ar@{.>}[d]^{\simeq} \nonumber \\
&\tilde{V} \otimes \tilde{Y} \nonumber
}
\eeq
where the bottom horizontal map is a trivial cofibration since those are preserved under pushout, $\tilde{U} \otimes \tilde{Y} \rarr \tilde{V} \otimes \tilde{Y}$ is a trivial cofibration for the same reason that $\tilde{U} \otimes \tilde{X} \rarr \tilde{V} \otimes \tilde{X}$ is a trivial cofibration, hence the dotted arrow is an equivalence by the 2-3 property. \\

Finally it suffices to show $\tilde{U} \otimes \tilde{Y} \coprod_{\tilde{U} \otimes \tilde{X}} \tilde{V} \otimes \tilde{X}$ is a cofibrant approximation to $U\otimes Y \coprod_{U \otimes X} V \otimes X$. It suffices to show $\tilde{A} \coprod_{\tilde{B}} \tilde{C} \rarrequ A \coprod_B C$ in the event that $\tilde{B} \rarr \tilde{C}$ corresponds to $\tilde{U} \otimes \tilde{X} \rarr \tilde{V} \otimes \tilde{X}$, trivial cofibration as shown above. Then referring to the commutative diagram below:
\beq
\xymatrix{
B \ar[dd] \ar[rr] &&C \ar[dd] \\
&\tilde{B} \ar[ul]_{\simeq} \ar[dd] \ar[rr] && \tilde{C} \ar[ul]_{\simeq} \ar[dd] \\
A \ar[rr] && A \coprod_B C \\
& \tilde{A} \ar[ul]_{\simeq} \ar[rr] && \tilde{A} \coprod_{\tilde{B}}\tilde{C} \ar@{.>}[ul]
} \nonumber
\eeq
$\tilde{A} \rarr \tilde{A} \coprod_{\tilde{B}} \tilde{C}$ is a trivial cofibration as a pushout so a weak equivalence in particular. In the top square $B \rarr C$ is an equivalence by the 2-3 property, $C \rarr A \coprod_B C$ is a cofibration, $\DCopw$ is left proper by Thm \ref{Thm 4.1.1}, so $A \rarr A \coprod_B C$ is an equivalence as pushout of an equivalence along a cofibration. In the bottom square it follows using the 2-3 property that $\tilde{A} \coprod_{\tilde{B}} \tilde{C} \rarr A \coprod_B C$ is an equivalence.\\

Then we also need that for any $X$ cofibrant in $\LGDCop$, $Q1 \otimes X \rarr 1 \otimes X$ is a weak equivalence in $\LGDCop$, i.e. an i$\Gamma$-local equivalence. $X$ cofibrant in $\DCopw$ implies that it is cofibrant in $\DCop$ as well. Then since $\DCop$ is a monoidal model category, $Q1 \otimes X \rarrequ 1 \otimes X$, but equivalences are also i$\Gamma$-local equivalences by Prop. \ref{3.1.5 mod}, hence it is a weak equivalence in $\LGDCop$ as well. We conclude $\LGDCop$ is a monoidal model category, with the same tensor product $\otimes$ as $\DCop$, and $\uHomDCopw = \uHomDCop$.

\subsection{Covers}
We regard $\cD_0$ as the first of a chain of proper cellular, symmetric monoidal model categories $\cD_i$, the idea being that we are interested in maps $\DiCop \rarr \DipCop$ and their localizations. This is especially pertinent in the setting of Segal topos, since internal Homs as defined presently very much depend on the target category. By having a tower of such target categories, one thereby pushes further the internal theories we have, producing a tower of internal theories, levelwise. Practically, it suffices to determine at which level in this tower one wishes to work. Lower-level theories are in a sense more compact by construction, while higher theories are clearly more encompassing.\\

Denote by $h_i: \cD_i \rarr \cD_{i+1}$ a map of model categories, that is, a left Quillen functor. It induces:
\beq
\hiCop: \DiCop \rarr \DipCop \nonumber
\eeq
pointwise. It follows from Prop. 11.6.5 of \cite{Hi} that since our model categories are cellular, in particular cofibrantly generated, this map is also a left Quillen functor. We refer the reader to Def. 8.5.11 of \cite{Hi} for defining the left derived functor $\bL F$ of a left adjoint $F$, part of a Quillen adjunction. Recall that $\cC$ is $\cD_0$-enriched, which means it is also $\cD_i$-enriched by composition with the $h_i$'s. Define:
\beq
h^{(i)} = h_{i-1} \circ \cdots \circ h_0 \nonumber
\eeq
and
\beq
h_x^{(i)} = h^{(i-1)} \circ h_x \nonumber
\eeq
Let $\Gamma_i$ be the following class of maps in $\DiCop$:
\beq
\Gamma_i = \{h_y^{(i)} \oT G \rarr h_x^{(i)} \oT G \; | \; x \rarr y \in W(\cCop),\, G \in \DiCop \} \nonumber
\eeq
Define:
\beq
\bL \hiCop \Gamma_i = \{\bL\hiCop(g) \; | \; g \in \Gamma_i \} \nonumber
\eeq
then by Thm 3.3.20 of \cite{Hi}, we have a left Quillen functor:
\beq
\hiCop: L_{\Gamma_i}\DiCop \rarr L_{\bL \hiCop \Gamma_i} \DipCop \nonumber
\eeq
Observe as an aside that we have $\bL \hiCop \Gamma_i \subset \Gamma_{i+1}$. Indeed, for $g \in \Gamma_i$:
\beq
\bL \hiCop (g) = \hiCop(\tilde{C}_i(g)) = h_i \circ \tilde{C}_i(g) \nonumber
\eeq
$\tilde{C}_i$ cofibrant approximation functor. But we may as well enlarge $\Gamma_i$ to also contain $\tilde{C}_i(g)$'s since we use internal homotopy function complexes, and their image by $h_i$ is therefore in $\Gamma_{i+1}$ by definition. It follows the collection of i$\bL \hiCop \Gamma_i$-local objects contains i$\Gamma_{i+1}$-local objects, and consequently i$\bL \hiCop \Gamma_i$-local equivalences are i$\Gamma_{i+1}$-local equivalences. We are looking at the following picture:
\beq
\xymatrix{
\cD_i \ar@{~>}[d] \ar[r]^{h_i} & \cD_{i+1} \ar@{~>}[d]\nonumber \\
\DiCop \ar[d]_{L_{\Gamma_i}} \ar[r]^{\hiCop} &\DipCop \ar[d]^{L_{\bL \hiCop \Gamma_i}} \nonumber \\
L_{\Gamma_i} \DiCop \ar[r]^{\hiCop} &L_{\bL \hiCop \Gamma_i} \DipCop \nonumber
}
\eeq
If we introduce the right Quillen adjoint $\kipCop$ of $\hiCop$, then one can define its right derived form following Def. 8.5.11 of \cite{Hi}, and Thm 3.3.20 again gives us that we have a Quillen adjunction (here we assume $L_{\Gamma_i}\DiCop$ is right proper):
\beq
R_{\bR \kipCop \cK_{i+1}}L_{\Gamma_i} \DiCop \stackrel{\hiCop}{\adj_{\kipCop}} R_{\cK_{i+1}} L_{\bL \hiCop \Gamma_i} \DipCop \nonumber
\eeq
induced from:
\beq
\xymatrix{
L_{\Gamma_i} \DiCop \ar[d]_{R_{\bR \kipCop \cK_{i+1}}} \ar[r]^{\hiCop} & L_{\bL \hiCop \Gamma_i} \DipCop \ar[d]^{R_{\cK_{i+1}}} \nonumber \\
R_{\bR \kipCop \cK_{i+1}}L_{\Gamma_i} \DiCop \ar[r]_{\hiCop} & R_{\cK_{i+1}} L_{\bL \hiCop \Gamma_i} \DipCop \nonumber
}
\eeq
where $K_i$ is a class of objects of $\DiCopw = L_{\Gamma_i}\cD_i^{\cC^{\op}}$ and \newline $\cK_i = \{i K_i -\text{colocal equivalences } \}$. This presupposes that the left localization of $\DiCop$ is right proper, cellular. That it is cellular is shown when we prove that we have an internal left Bousfield localization, but whether it is right proper is not automatic. This is something we have to assume. In other terms this construction holds for symmetric monoidal model categories $\DiCop$ whose internal left Bousfield localizations are right proper.

\section{Foundational results}
\subsection{Elementary results}
We list a few results that will be useful in the sequel. First we have:
\beq
\uHom(1,A) \cong A \nonumber
\eeq
since for any $W$, $\Hom(W, \uHom(1,A)) \cong \Hom(W \otimes 1, A) \cong \Hom(W,A)$. We also have:
\beq
\Hom(1, \uHom(A,B)) \cong \Hom(A,B) \nonumber
\eeq
as it directly follows from the adjunction isomorphism $\Hom(1, \uHom(A,B)) \cong \Hom(1 \otimes A , B) \cong \Hom(A,B)$. From this we have the very useful:
\begin{equ} \label{3.1.1}
Internal equivalences $\uHom(\tilde{A},\hat{B}) \rarrequ \uHom(\tilde{A},\hat{C})$ imply classical equivalences $\Hom(\tilde{A},\hat{B}) \rarrequ \Hom(\tilde{A},\hat{C})$.
\end{equ}
\begin{proof}
It suffices to write:
\beq
\xymatrix{
\Hom(1, \uHom(\tilde{A},\hat{B})) \ar@{=}[d] \ar[r]^{\simeq} & \Hom(1, \uHom(\tilde{A},\hat{C})) \ar@{=}[d] \nonumber \\
	\Hom(\tilde{A},\hat{B}) \ar@{.>}[r]_{\simeq} &\Hom(\tilde{A}, \hat{C})
} \label{gequ implies equ}
\eeq
\end{proof}

A very important result is the following, a modified version of Thm 17.7.7 of \cite{Hi}:
\begin{17.7.7} \label{17.7.7 mod}
Let $g: X \rarr Y $ a map in $\DCop$. Then $g$ is a weak equivalence if and only if for all $W$ $g_*: \uHom(\tilde{W}, \hat{X}) \rarr \uHom(\tilde{W},\hat{Y})$ is a weak equivalence in $\DCopDop$, if and only if for all $Z$ in $\DCop$, $g^*: \uHom(\tilde{Y},\hat{Z}) \rarr \uHom(\tilde{X}, \hat{Z})$ is a weak equivalences in $\DCopDop$.
\end{17.7.7}
\newpage
\begin{proof}
Suppose $g:X \rarr Y$ is an equivalence. Then by the original Thm 17.7.7 of \cite{Hi}, we have the following top equivalence for any objects $W$ and $A$:
\beq
\xymatrix{
\Hom(\widetilde{W \otimes A}, \hat{X}) \ar@{=}[d] \ar[r]^{\simeq} &\Hom(\widetilde{W \otimes A}, \hat{Y}) \ar@{=}[d]\\
\Hom(\tilde{W} \otimes \tilde{A}, \hat{X}) \ar[d]_{\cong} \ar[r]^{\simeq} & \Hom(\tilde{W} \otimes \tilde{A}, \hat{Y}) \ar[d]^{\cong}\\
\Hom(\tilde{W}, \uHom(\tilde{A}, \hat{X})) \ar@{.>}[r]_{\simeq} & \Hom(\tilde{W}, \uHom(\tilde{A}, \hat{Y}))
} \nonumber
\eeq
and this being true for any $W$ it follows $\uHom(\tilde{A}, \hat{X}) \rarrequ \uHom(\tilde{A}, \hat{Y})$. Conversely, if this is true, that implies $\Hom(\tilde{A}, \hat{X}) \rarrequ \Hom(\tilde{A}, \hat{Y})$ by Prop. \ref{3.1.1}, hence an equivalence $g: X \rarrequ Y$ by the original Thm 17.7.7. For the second part of the theorem we proceed in like manner. For any $W$ and $Z$ we have:
\beq
\xymatrix{
\Hom(\widetilde{W \otimes Y}, \hat{Z}) \ar@{=}[d] \ar[r]^{\simeq} & \Hom(\widetilde{W \otimes X}, \hat{Z}) \ar@{=}[d] \nonumber \\
\Hom(\tilde{W} \otimes \tilde{Y}, \hat{Z}) \ar[d]^{\cong}  &\Hom(\tilde{W} \otimes \tilde{X}, \hat{Z}) \ar[d]^{\cong} \nonumber \\
\Hom(\tilde{W}, \uHom(\tilde{Y}, \hat{Z})) \ar@{.>}[r]^{\simeq} &\Hom(\tilde{W}, \uHom(\tilde{X}, \hat{Z})) \nonumber
}
\eeq
The reasoning for showing that $g$ is an equivalence if and only if $g^*$ is an equivalence is identical to the one above for showing that $g$ is an equivalence if and only if $g_*$ is one as well.
\end{proof}

\subsection{Yoneda}
We take the left Bousfield localization of $\DCop$ with respect to $\Gamma = \Gamma^0 \oT \DCop$, with $\Gamma^0 = \{ h_y \rarr h_x \; | \; x \rarr y \in W(\cCop) \}$. In particular, $F \in \DCop$ i$\Gamma$-local (see Def. \ref{Glocal}) is therefore i$\Gamma^0$-local, that is it satisfies:
\beq
\uHomDCop(\tilde{h_x}, \hat{F}) \rarrequ \uHomDCop(\tilde{h_y}, \hat{F}) \nonumber
\eeq
Recall that we initially wanted to have functors from $\cCop$ to $\cD_0$ that map equivalences to equivalences. To show that such local objects $F$ as just defined preserve equivalences will in a first time necessitate a Yoneda lemma for enriched functors. We use \cite{K} and \cite{McL}. Consider a $\cD_0$-functor $F: \cCop \rarr \cD_0$. Let $x \in \cC$. Then the morphism:
\beq
\phi_y: Fx \rarr \uHomD(h_x(y),Fy) \nonumber
\eeq
expresses $Fx$ as the end $\int_{y \in \cCop}\uHomD(h_x(y),Fy)$. Observe that we could have obtained this result using the $\cD_0$-module structure on $\DCop$. We have:
\beq
\map_{\DCop}(h_x,F) = \int_{y \in \cCop}\uHomD(h_x(y),Fy) \nonumber
\eeq
but we also have, with \eqref{intHomDCop}:
\begin{align}
\map_{\DCop}(h_x,F)&=\map_{\DCop}(h_x \otimes 1, F) \nonumber\\
&=\uHomDCop(1,F)(x) \nonumber \\
&=F(x) \nonumber
\end{align}
hence:
\beq
F(x) = \int_{y \in \cCop} \uHomD(h_x(y), Fy) \nonumber
\eeq
We improve on this result since we still have to connect this to an iRhfC.\\

\begin{Yoneda}
\beq
Fx \cong \int_{y \in \cCop}\uHomD(\tilde{h}_x(y),Fy) \label{enrichedYoneda}
\eeq
\end{Yoneda}
\begin{proof}
This is an adaptation of the proof found in \cite{K}. Consider any $\cD_0$-natural map $ \alpha_y^c:X \rarr \uHom(\tilde{h}_x(y),Fy)$ whose adjoint we will represent by the same letter: $\alpha_y^c: \tilde{h}_x(y) \rarr \uHom(X,Fy)$. We have $\phi_y^c: Fx \rarr \uHom(\tilde{h}_x(y),Fy)$, with adjoint again represented by the same letter, $\phi_y^c: \tilde{h}_x(y) \rarr \uHom(Fx,Fy)$. We want to show there is a unique $\eta: X \rarr Fx$ such that $\alpha_y^c = \phi_y^c \eta$, and that will prove our result. In the adjoint picture, that amounts to showing there exists a unique $\eta$ such that the triangular diagram below commutes:
\beq
	\xymatrixcolsep{4pc}
\xymatrix{
\alpha_y: h_x(y) \ar[r] &\uHom(Fx,Fy) \ar[r]^-{\uHom(\eta,Fy)} & \uHom(X,Fy) \nonumber \\
\tilde{h}_x(y) \ar[u]_{\simeq} \ar@{.>}[ur]^{\phi_y^c} \ar@{.>}[urr]_{\alpha_y^c}
}
\eeq
but from the classical proof we know there is such an $\eta$, and it is unique, hence by composition we have our result.

\end{proof}

\subsection{Localization with respect to $\Gamma$ produces prestacks}
Our aim is to have functors $F: \cCop \rarr \cD_0$ that map equivalences to equivalences, for the sake of obtaining a notion of prestacks. This can be implemented using a Bousfield localization with respect to $\Gamma^0 = \{ h_y \rarr h_x \; | \; x \rarr y \in W(\cCop) \}$. To make this manifest, we needed an enriched Yoneda lemma that we covered in the previous section. With regards to the collection of maps we are using as base for our localization, note that we will localize with respect to $\Gamma = \Gamma^0 \oT \DCop$ instead, which will in particular imply a localization with respect to $\Gamma^0$. Some technical lemmas later will require such an enlargement. For the moment, a heuristic motivation for using $\Gamma$ goes as follows: $\Gamma^0$-local objects recognize morphisms $h_y \rarr h_x$ as weak equivalences via Hom objects. It is not much of a stretch to ask that $\Gamma^0$-local objects $F$ also satisfy $\Hom(\tilde{h_x} \oT \tilde{C}, \hat{F}) \rarrequ \Hom(\tilde{h_y} \oT \tilde{C}, \hat{F})$ for all $C$ as well, hence we consider $\Gamma$-local objects instead.\\ 

$F: \cCop \rarr \cD_0$ i$\Gamma$-local is i$\Gamma^0$-local, which means that for any element of $\Gamma^0$, we have:
\beq
\uHomDCop(\tilde{h}_x, \hat{F}) \rarrequ \uHomDCop(\tilde{h}_y, \hat{F}) \label{Gammalocal}
\eeq
where $\tilde{h}$ is a cofibrant approximation of $h$, and $\hat{F}$ is a simplicial resolution of $F$. Recall that for simplicial objects $\hat{X}$, we define the simplicial set $\uHom(A,\hat{X})$, as in \cite{Hi}:
\beq
\uHom(A,\hat{X})_n = \uHom(A,\hat{X}_n) \nonumber
\eeq
We show \eqref{Gammalocal} implies $Fx \rarrequ Fy$ if $x \rarr y$ is in $W$ for $F$ i$\Gamma^0$-local. Since $\uHomDCop(\tilde{h}_x, \hat{F}) \rarrequ \uHomDCop(\tilde{h}_y, \hat{F})$ is an equivalence in the Reedy model structure of $\DCopDop$, this means for all $[n] \in \Delta$, we have:
\beq
\xymatrix{
\uHomDCop(\tilde{h}_x, \hat{F})_n \ar@{=}[d] \ar[r]^{\simeq} & \uHomDCop(\tilde{h}_y, \hat{F})_n \
\ar@{=}[d]  \\
\uHomDCop(\tilde{h}_x, \hat{F}_n) \ar@{.>}[r]^{\simeq} &\uHomDCop(\tilde{h}_y, \hat{F}_n)
} \label{locwrtG}
\eeq
$\hat{F}$ being a simplicial resolution of $F$ means $cs_*(F) \rarrequ \hat{F}$, where $cs_*$ is the constant simplicial functor. This being a Reedy weak equivalence in $\DCopDop$, it follows that for all $[n] \in \Delta$, we have equivalences in $\DCop$: $F = (cs_*(F))_n \rarrequ \hat{F}_n$. Using this fact, along with \eqref{locwrtG}, \eqref{daguer} and the enriched Yoneda lemma \eqref{enrichedYoneda}, it follows that we have:
\beq
\xymatrix{
\hat{F}_n(x) \cong \int_{t \in \cCop}\uHomD(\tilde{h}_x(t), \hat{F}_n(t)) \ar[r]^{\simeq} &\hat{F}_n(y) \cong \int_{t \in \cCop}\uHomD(\tilde{h}_y(t), \hat{F}_n(t)) \nonumber \\
F(x) \ar[u]^{\simeq}_{\cD_0} \ar@{.>}[r]_{\simeq} &F(y)  \ar[u]^{\simeq}_{\cD_0}
}
\eeq
where the bottom equivalence follows from the 2-3 property. This shows i$\Gamma$-local objects map equivalences to equivalences.\\

\subsection{ Internal Hom is fibrant}
Crucial in all our work is the fact that internal homotopy function complexes are fibrant objects. This is needed to use Thm 17.7.7 to prove its modified version, Thm \ref{17.7.7 mod}. We state this as a result:
\begin{Fibrant}
Let $C$ be cofibrant in $\DCop$, $\hat{X}$ a simplicial resolution of $X \in \DCop$. Then $\uHomDCop(C, \hat{X})$ is fibrant in the Reedy model structure of $\DCopDop$.
\end{Fibrant}
\begin{proof}
$\uHom(C, \hat{X})$ fibrant means $\uHom(C,\hat{X}) \rarr *$ is a fibration in $\DCopDop$, where $*$ denotes the terminal object of $\DCopDop$. This is true if for all $[n] \in \Delta$, we have:
\beq
\uHom(C,\hat{X})_n = \uHom(C,\hat{X}_n) \rarr * \times_{M_n *} M_n \uHom(C, \hat{X}) \nonumber
\eeq
is a fibration in $\DCop$, where $M_nA$ is the $n$-th matching object of $A$ (see \cite{Hi}). Given the definition of the matching object, the object on the right of this map simplifies as $M_n\uHom(C, \hat{X})$, which is equal to $\uHom(C, M_n \hat{X})$, since the matching object is a limit, and the internal Hom commutes with limits. Thus we seek to show that:
\beq
\uHom(C, \hat{X}_n) \rarr \uHom(C, M_n \hat{X}) \nonumber
\eeq
is a fibration in $\DCop$. Let $D \rarr E$ be a trivial cofibration in $\DCop$. We need a lift $\alpha$ in the commutative diagram below:
\beq
\xymatrix{
D \ar[d]_{\substack{triv \\ cof}} \ar[r] &\uHom(C,\hat{X}_n) \ar[d] \nonumber \\
E \ar@{.>}[ur]^{\alpha} \ar[r] &\uHom(C, M_n \hat{X}) \nonumber
}
\eeq
by adjunction this is equivalent to having a lift in the diagram below:
\beq
\xymatrix{
D \otimes C \ar[d] \ar[r] &X_n \ar[d]\nonumber \\
E \otimes C \ar@{.>}[ur] \ar[r] &M_n \hat{X} \nonumber
}
\eeq
but $C$ is cofibrant, $\otimes$ is a left Quillen functor, so $D \otimes C \rarr E \otimes C$ is a trivial cofibration, and $\hat{X}$ being a simplicial resolution, it is fibrant in $\DCopDop$, so we have such a lift, which completes the proof.
\end{proof}
The other internal homotopy function complex we work with is the one we consider after having taken a left Bousfield localization with respect to $\Gamma$, and we find ourselves in $\DCopw = L_{\Gamma}\DCop$. The iLhfC of interest is now $\uHomDCopw(\tilde{k_0}, \hat{F})$, where $k_0 \in K_0$, $\tilde{k_0}$ is a cosimplicial resolution of $k_0$, and $\hat{F}$ is a fibrant approximation to $F$. As a preliminary result, we prove:
\begin{higheradj} \label{3.4.2}
	For $\tilde{D}$ a cofibrant approximation to $D$ in $\DCopw$, $\hat{X}$ a fibrant approximation to $X$, $\tilde{k_0}$ a cosimplicial resolution of $k_0$, we have:
\beq
\HomDCopw(\tilde{D}, \uHomDCopw(\tilde{k_0}, \hat{X})) \cong \HomDCopw(\tilde{D} \otimes \tilde{k_0}, \hat{X}) \nonumber
\eeq
\end{higheradj}
\begin{proof}
Since for $\tilde{k_0}$ a cosimplicial resolution, we have $\uHom(\tilde{k_0}, A)_n = \uHom(\tilde{k_0}_n, A)$, and for $\hat{W}$ a simplicial resolution we have $\Hom(Y,\hat{W})_n = \Hom(Y,\hat{W}_n)$, it suffices to show this isomorphism of simplicial sets on components:
\begin{align}
\Hom(\tilde{D}, \uHom(\tilde{k_0}, \hat{X}))_n & =\Hom(\tilde{D}, \uHom(\tilde{k_0}, \hat{X})_n) \nonumber \\
&=\Hom(\tilde{D}, \uHom(\tilde{k_0}_n, \hat{X})) \nonumber \\
&\cong \Hom(\tilde{D} \otimes  \tilde{k_0}_n, \hat{X}) \nonumber \\
&=\Hom((\tilde{D} \otimes  \tilde{k_0})_n, \hat{X}) \nonumber \\
&=\Hom(\tilde{D} \otimes  \tilde{k_0}, \hat{X})_n \nonumber
\end{align}
\end{proof}
\begin{Fibrant2}
For $\tilde{k}_0$ a cofibrant approximation to $k_0$, $\hat{X}$ a simplicial resolution of $X$, objects of $\DCopw$, $\uHomDCopw(\tilde{k}_0, \hat{X})$ is fibrant in the Reedy model structure of $\DCopwDop$.
\end{Fibrant2}
\begin{proof}
It follows from the adjunction isomorphism, that the adjoint to:
\beq
\xymatrix{
h_0 \ar[d] \ar[r] &\uHom(\tilde{k_0}, \hat{X}) \ar[d]\nonumber \\
l_0 \ar[r] &\uHom(\tilde{t_0}, \hat{X}) \times_{\uHom(\tilde{t_0}, \hat{Y})} \uHom(\tilde{k_0}, \hat{Y}) \nonumber
}
\eeq
is the following diagram:
\beq
\xymatrix{
\tilde{t_0} \otimes l_0 \coprod_{\tilde{t_0} \otimes h_0} \tilde{k_0} \otimes h_0 \ar[d] \ar[r] &\hat{X} \ar[d] \nonumber \\
\tilde{k_0} \otimes l_0 \ar[r] &\hat{Y} \nonumber
}
\eeq
In particular if $L_n$ denotes the $n$th-latching object functor,
\beq
\xymatrix{
\tilde{C} \ar[d] \ar[r] &\uHom(\tilde{k_{0,n}}, \hat{X}) \ar[d]\nonumber \\
\tilde{D} \ar[r] &\uHom(L_n\tilde{k_0}, \hat{X}) \times_{\uHom(L_n\tilde{k_0}, \hat{Y})} \uHom(\tilde{k_{0,n}}, \hat{Y}) \nonumber
}
\eeq
which simplifies to:
\beq
\xymatrix{
\tilde{C} \ar[d] \ar[r] &\uHom(\tilde{k_0}, \hat{X}) \ar[d]\nonumber \\
\tilde{D} \ar[r] &M_n \uHom(\tilde{k_0}, \hat{X}) \nonumber
}
\eeq
if $\hat{Y} = *$, has:
\beq
\xymatrix{
L_n\tilde{k_0} \otimes \tilde{D} \coprod_{L_n\tilde{k_0} \otimes \tilde{C}} \tilde{k_{0,n}} \otimes \tilde{C} \ar[d] \ar[r] &\hat{X} \ar[d]  \\
\tilde{k_{0,n}} \otimes \tilde{D} \ar[r] &\star
}  \nonumber
\eeq
for adjoint. Now the map on the right is a fibration. If the one on the left is a trivial cofibration, we have a lift, and that would prove our claim. But $\tilde{C} \rarr \tilde{D}$ is a cofibration, $\otimes$ is a Quillen bifunctor since $\DCopwDop$ is a symmetric monoidal model category as argued in Section 2.2.4, hence $\tilde{k_0}$ being cofibrant in $\DCopwDop$ as a cosimplicial resolution, the functor $\tilde{k_0} \otimes - : \DCopw \rarr \DCopwDop$ is a left Quillen functor, so preserves trivial cofibrations, hence $\tilde{k_0} \otimes \tilde{C} \rarr \tilde{k_0} \otimes \tilde{D}$ is a trivial cofibration in $\DCopwDop$, which means exactly that the left vertical map in the above commutative diagram is a trivial cofibration in $\DCopw$ by definition of cofibrations in Reedy model categories. This completes our proof.
\end{proof}

\subsection{Simplicial resolution} \label{simpres}
In this subsection, we show the following fact which is implied in the proof of Prop. \ref{17.7.7 mod}:
\begin{HomAXsimpres}
If $\hat{X}$ is a simplicial resolution of $X$, $\uHom(A,\hat{X})$ is a simplicial resolution of $\uHom(A,X)$.
\end{HomAXsimpres}
\begin{proof}	
This follows from Prop. 17.4.16 of \cite{Hi}, which states that if $F: \cM \rlarr \cN:G$ is a Quillen adjunction, $X$ is cofibrant in $\cM$, $Y$ is fibrant in $\cN$, $\hat{Y}$ is a simplicial resolution of $Y$, then $G(\hat{Y})$ is a simplicial resolution of $G(Y)$. We apply this to the case:
\beq
-\otimes A: \DCop \rlarr \DCop: \uHom(A,-)=G \nonumber
\eeq
Let $\hat{Y}$ be a simplicial resolution of $Y$ fibrant. Then $G(\hat{Y}) = \uHom(A,\hat{Y})$ is a simplicial resolution of $G(Y) = \uHom(A,Y)$. For any $Y$, take a fibrant replacement of $\uHom(A,Y)$:
\begin{align}
\uHom(A,Y) \rarrequ R\uHom(A,Y) &=RG(Y) \nonumber \\
&=G(RY) = \uHom(A,RY) \nonumber
\end{align}
so that we have $\uHom(A,Y) \rarrequ \uHom(A,RY)$, which implies:
\beq
cs_* \uHom(A,Y) \rarrequ cs_* \uHom(A,RY) \rarrequ \uHom(A,\hat{Y}) \nonumber
\eeq
Here $\hat{Y}$ is a simplicial resolution of $RY$. But $Y \rarrequ RY$ implies $cs_*Y \rarrequ cs_* RY \rarrequ \hat{Y}$, so $\hat{Y}$ is also a simplicial resolution of $Y$. For such a simplicial resolution, we have $\uHom(A,\hat{Y})$ is a simplicial resolution of $\uHom(A,Y)$ as shown above. This is independent of the choice of simplicial resolution since all such resolutions are Reedy equivalent by Lemma 16.1.22 of \cite{Hi}.
\end{proof}
In the same manner we would show that if $\tilde{Y}$ is a cosimplicial resolution of $Y$, then $\uHom(\tilde{Y},B)$ is a simplicial resolution of $\uHom(Y,B)$.\\

\subsection{Equivalence of internal homotopy function complexes}
Our internal homotopy function complexes are defined as the homotopy function complexes of \cite{Hi}, save that instead of using Hom sets, we use internal Homs. We want those internal homotopy function complexes to be independent of the choice of resolutions and approximations used in defining them. We first need a couple of definitions, variants of those found in \cite{Hi}:
\begin{17.2.5}
	A \textbf{change of iRhfC map}: \index{change! of iRhfc map}
\beq
(\tilde{X}, \hat{Y}, \uHom(\tilde{X},\hat{Y})) \rarr (\tilde{X'}, \hat{Y'}, \uHom(\tilde{X'},\hat{Y'})) \nonumber
\eeq
is a triple $(f,g,h)$ formed of a map of cofibrant approximations $f: \tilde{X} \rarr \tilde{X'}$, a map of simplicial resolutions $g: \hat{Y} \rarr \hat{Y'}$ and the map of simplicial objects $h: \uHom(\tilde{X}, \hat{Y}) \rarr \uHom(\tilde{X'}, \hat{Y'})$ induced by $f$ and $g$.
\end{17.2.5}
\begin{17.2.7}
	For $X,Y \in \DCop$, we define the \textbf{category $iRhfC(X,Y)$} to be the category of iRhfC's from $X$ to $Y$ and with changes of iRhfC maps as morphisms. \label{iRhfCXY}
\end{17.2.7}
\begin{17.2.11}
Let $X,Y \in \DCop$. Then any two iRhfC's from $X$ to $Y$ are connected by an essentially unique zig-zag of changes of iRhfC maps.
\end{17.2.11}
\begin{proof}
This follows from Thm 14.4.5 of \cite{Hi}, which states that if $\cC$ is a category, $X,Y \in \cC$, $B\cC$ is contractible, then there exists an essentially unique zig-zag from $X$ to $Y$ in $\cC$, and the proposition below, a modified version of Prop. 17.2.10 of \cite{Hi} to the internal setting.
\end{proof}
\begin{17.2.10}
Let $X,Y \in \DCop$. Then $B iRhfC(X,Y)$ is contractible.
\end{17.2.10}
\begin{proof}
We have $B iRhfC(X,Y) = B sRes(\uHom(\tilde{X},Y))$ by Section \ref{simpres}, where $sRes$ stands for simplicial resolution, and this latter category is contractible by Prop. 16.1.5 of \cite{Hi}.
\end{proof}

\newpage

\section{The cardinal $\gamma$ in the proof of Prop. 4.5.1}
To prove that we have an internal left Bousfield localization, we use Thm 11.3.1 of \cite{Hi}, which itself needs Prop. 4.5.1 that we generalize to our setting. The proof of the latter proposition in the classical case uses a cardinal $\gamma$. Following Def. 4.5.3 of \cite{Hi}, $\gamma = \xi^{\xi}$, where $\xi$ is the smallest cardinal that is at least as large as any of the cardinals that are the subject of the following five sections. We use the same definitions in the generalized case.\\

\subsection{Size of the cells of $\DCop$}
By definition, the size of the cells of a cellular model category $\cM$ is the smallest cardinal for which Thm 12.3.1 of \cite{Hi} holds. This theorem makes no use of a notion of equivalence, and can be used as is, hence holds also in the internal setting. Hence we can define the size of the cells of $\DCop$ following that result.\\

\subsection{Compactness of the domains of $I$}
$\DCop$ is a cellular model category, in particular cofibrantly generated, and let $I$ denote its set of generating cofibrations. Let $\eta$, as in \cite{Hi}, be a cardinal such that the domains of elements of $I$ are $\eta$-compact.\\

\subsection{Cardinal $\lambda$ in the proof of Thm 4.3.1}
As in \cite{Hi}, let $\lambda$ be the cardinal used in the proof of Thm 4.3.1. This result invokes a set $\wLG$, originally introduced in Prop. 4.2.5 of \cite{Hi}, which uses equivalences, and therefore needs to be stated and proved in the internal setting. To prove it, we invoke the equivalence between $\Gamma$-local equivalences and i$\Gamma$-local equivalences (Lemma \ref{lequ implies glequ}). In the proof of the original Proposition, one result is interesting in its own right, and we prove it in the internal case, it is Prop. \ref{3.1.5 mod} below. We need it in Section 2.2.5.\\

\newpage

\begin{3.1.5} \label{3.1.5 mod}
Let $\Gamma$ be a class of maps in $\DCop$. Then any weak equivalence in $\DCop$ is also a i$\Gamma$-local equivalence.
\end{3.1.5}
\begin{proof}
Let $A \rarr B$ be a weak equivalence in $\DCop$. If $X$ is an i$\Gamma$-local object, we want $\uHom(\tilde{B}, \hat{X}) \rarrequ \uHom(\tilde{A}, \hat{X})$, $\tilde{A}$ a cofibrant approximation to $A$, and $\hat{X}$ a simplicial resolution to $X$. For $C$ cofibrant in $\DCop$, we have:
\beq
\xymatrix{
\Hom(C, \uHom(\tilde{B}, \hat{X})) \ar[d]^{\cong} \ar[r] &\Hom(C, \uHom(\tilde{A}, \hat{X})) \ar[d]^{\cong} \nonumber \\
\Hom(C \otimes \tilde{B}, \hat{X})  \ar[r] & \Hom(C \otimes \tilde{A}, \hat{X})  \nonumber \\
}
\eeq
Now we use the fact that $\otimes$ being a left Quillen functor, $C$ being cofibrant, if $\tilde{A} \rarr \tilde{B}$ is a fibrant cofibrant approximation to $A \rarr B$ that is a cofibration, then $C \otimes \tilde{A} \rarr C \otimes \tilde{B}$ is a trivial cofibration as well, in particular it is a weak equivalence, so by the original result of \cite{Hi}, it is a $\Gamma$-local equivalence, so that the bottom horizontal map above is an equivalence, since $X$ i$\Gamma$-local is also $\Gamma$-local by Lemma \ref{gloc implies loc}. It follows from the above commutative diagram and Thm 17.7.7 of \cite{Hi} that we have an equivalence $\uHom(\tilde{B}, \hat{X}) \rarr \uHom(\tilde{A}, \hat{X})$ in $\DCopDop$, that is $A \rarr B$ is an i$\Gamma$-local equivalence.
\end{proof}

\begin{4.2.5} \label{4.2.5 mod}
If $I$ denotes the set of generating cofibrations of $\DCop$, $\Gamma$ is a class of maps in $\DCop$, then there exists a set $\wLG$ of relative $I$-cell complexes whose domains are cofibrant, such that every element of $\wLG$ is an i$\Gamma$-local equivalence, and an object $W$ is i$\Gamma$-local if and only if $W \rarr *$ is a $\wLG$-injective.
\end{4.2.5}
\begin{proof}
By Lemmas \ref{gloc implies loc} and \ref{lequ implies glequ}, this is equivalent to the original Proposition.
\end{proof}

The cardinal in Thm 4.3.1 of \cite{Hi} is $\lambda = succ(\kappa)$, $\kappa$ a cardinal which according to Lemma 10.4.6 of \cite{Hi} is such that the domain of every element of $\wLG$ is $\kappa$-small relative to the subcategory of relative $\wLG$-cell complexes. This is a classical notion and needs not be generalized.\\

\newpage

\subsection{Cardinal $\kappa$ in Prop. 12.5.3 of \cite{Hi}}
This proposition is applied to the set $\wLG$. It mentions a cardinal $\kappa$ at least as large as four kinds of cardinals, two of which are cardinals given by Prop. 12.5.2 of \cite{Hi}, which makes use of a Hom set. We generalize this Proposition presently:
\begin{12.5.2} \label{12.5.2 mod}
If $X$ is a cofibrant object of $\DCop$, then there is a cardinal $\eta$ such that for $\nu \geq 2$ a cardinal, $Y$ a cell complex of size $\nu$, $\uHom(X,Y)$ has cardinality at most $\nu^{\eta}$.
\end{12.5.2}
\begin{proof}
Let $C$ be cofibrant in $\DCop$. We have $\Hom(C, \uHom(X,Y)) \cong \Hom(C \otimes X, Y)$. Since we are in a monoidal model category, $C \otimes X$ is again cofibrant. We apply the original Proposition of \cite{Hi} to $C \otimes X$, cofibrant, and $Y$, which gives $size(\Hom(C \otimes X, Y)) \leq \nu^{\eta}$. Finally, $size(\uHom(X,Y)) < size(\Hom(C, \uHom(X,Y)) = size(\Hom(C \otimes X, Y))$, which completes the proof.
\end{proof}

\subsection{Cardinal $\kappa$ in Prop. 12.5.7 of \cite{Hi}}
$\kappa$ is an infinite cardinal at least as large as four types of cardinals, two of which are given by Prop. \ref{12.5.2 mod} above, and one of which is given by Def. 12.5.5 of \cite{Hi}, which invokes a smallness argument, hence does not need to be modified.

\newpage

\section{Results needed for an internal left Bousfield Localization}

The following result is the first one that is needed to prove that we do have an internal left Bousfield localization:\\
\begin{3.2.3} \label{3.2.3 mod}
For $\Gamma$ a class of maps in $\DCop$, the class of i$\Gamma$-local equivalences satisfies the 2-3 property.
\end{3.2.3}
\begin{proof}
This is just a generalization of Hirschhorn's proof in \cite{Hi}: let $g:X \rarr Y$, $h:Y \rarr Z$ be maps, apply a functorial cofibrant factorization to those:
\beq
\xymatrix{
\tilde{X} \ar[d] \ar[r]^{\tilde{g}} &\tilde{Y} \ar[d] \ar[r]^{\tilde{h}} & \tilde{Z} \ar[d] \nonumber \\
X \ar[r]^{g} &Y \ar[r]^{h} &Z \nonumber
}
\eeq
where $\tilde{g}$, $\tilde{h}$ and $\tilde{h}\tilde{g}$ are cofibrant approximations to $g$, $h$ and $hg$ respectively. Those exist by virtue of Prop. 8.1.23 of \cite{Hi}. Let $W$ be an i$\Gamma$-local object, $\hat{W}$ a simplicial resolution of $W$. To say for example that $g$ is an i$\Gamma$-local equivalence would mean:
\beq
\uHomDCop(\tilde{Y},\hat{W}) \rarrequ \uHomDCop(\tilde{X}, \hat{W}) \nonumber
\eeq
is an equivalence in $\DCopDop$, where equivalences satisfy the 2-3 property, so if two of $\tilde{g}^*: \uHom(\tilde{Y}, \hat{W}) \rarr \uHom(\tilde{X}, \hat{W})$, $\tilde{h}^*: \uHom(\tilde{Z}, \hat{W}) \rarr \uHom(\tilde{Y}, \hat{W})$ or $(\tilde{h}\tilde{g})^*: \uHom(\tilde{Z}, \hat{W}) \rarr \uHom( \tilde{X}, \hat{W})$ is a weak equivalence, so is the third, which completes the proof.
\end{proof}

In the statement of Thm 11.3.1 of \cite{Hi}, mention is made of a set $J$, which exists by virtue of Prop. 4.5.1 of the same reference, which we generalize presently. Its proof uses three results of \cite{Hi}, two of which make use of a notion of equivalence, and therefore have to be generalized. Their proof needs the following definition, along with two lemmas, Lemma \ref{gloc implies loc} and Lemma \ref{lequ implies glequ} which we state and prove after those two results.
\begin{ideal}
	A class of maps $S$ in a symmetric monoidal model category $\cM$ is said to be an \textbf{ideal class} \index{map! ideal class} of maps if for all $f:A \rarr B$ in $S$, for all object $C$ of $\cM$, $C \otimes A \rarr C \otimes B$ is also in $S$.
\end{ideal}
Observe that the set of i$\Gamma$-local equivalences and the set of $\Gamma$-local equivalences form ideal classes:\\

\begin{iGammalocequideal}
The class of i$\Gamma$-local equivalences forms an ideal class.
\end{iGammalocequideal}
\begin{proof}
Let $f:X \rarr Y$ be an i$\Gamma$-local equivalence, let $C$ be any object. We show $C \otimes f$ is also an i$\Gamma$-local equivalence, that is if $W$ is an i$\Gamma$-local object, $\uHom(\tilde{Y}, \hat{W}) \rarrequ \uHom(\tilde{X}, \hat{W})$ implies $\uHom(\tilde{C} \otimes \tilde{Y}, \hat{W}) \rarrequ \uHom(\tilde{C} \otimes \tilde{X}, \hat{W})$. Let $Z$ be any object. We have:
\beq
\xymatrix{
	\Hom(\tilde{Z}, \uHom(\tilde{C} \otimes \tilde{Y}, \hat{W})) \ar[d]^{\cong} \ar@{.>}[ddddr]_{\simeq}  \\
\Hom(\tilde{Z} \otimes \tilde{C} \otimes \tilde{Y}, \hat{W}) \ar[r]^{\cong} & \Hom(\tilde{Z} \otimes \tilde{C}, \uHom(\tilde{Y}, \hat{W})) \ar[d]^{\simeq} \\
&\Hom(\tilde{Z} \otimes \tilde{C}, \uHom(\tilde{X}, \hat{W})) \ar[d]^{\cong} \\
	&\Hom(\tilde{Z} \oT \tilde{C} \oT \tilde{X}, \hat{W}) \ar[d]^{\cong} \\
&\Hom(\tilde{Z}, \uHom(\tilde{C} \otimes \tilde{X}, \hat{W}))
} \nonumber
\eeq
and this for any $Z$ shows by Thm 17.7.7 of \cite{Hi} that we have our desired equivalence, hence $C \otimes f$ is an i$\Gamma$-local equivalence.
\end{proof}
The following lemma is the first one that necessitates that we use $\Gamma$ as collection of morphisms for localization purposes, and not simply $\Gamma^0$.
\begin{Gammalocequideal}
The class of $\Gamma$-local equivalences forms an ideal class.
\end{Gammalocequideal}
\begin{proof}
Let $f:X \rarr Y$ be a $\Gamma$-local equivalence, $C$ any object of $\DCop$. We show $C \oT f$ is also a $\Gamma$-local equivalence, that is for any $W$ $\Gamma$-local object, we have $\Hom(\widetilde{C \oT Y}, \hat{W}) \rarrequ \Hom(\widetilde{C \oT X}, \hat{W})$. Since $\Hom(\widetilde{C \oT X},\hat{W}) \cong \Hom(\tilde{X}, \uHom(\tilde{C}, \hat{W}))$, it suffices to show $\uHom(\tilde{C},\hat{W})$ is a $\Gamma$-local object for all $C$, and $W$ $\Gamma$-local. Consider:
	\beq
	\xymatrix{
		\Hom(\tilde{h_x} \oT \tilde{D}, \uHom(\tilde{C}, \hat{W})) \ar[d]_{\cong} \ar[r] & \Hom(\tilde{h_y} \oT \tilde{D}, \uHom(\tilde{C}, \hat{W})) \ar[d]^{\cong} \\
		\Hom(\tilde{h_x} \oT \tilde{D} \oT \tilde{C}, \hat{W}) \ar[d]_{\cong} \ar[r] & \Hom(\tilde{h_y} \oT \tilde{D} \oT \tilde{C}, \hat{W}) \ar[d]^{\cong} \\
	\Hom(\tilde{C}, \uHom(\tilde{h_x} \oT \tilde{D}, \hat{W})) \ar[r]^{\simeq} & \Hom(\tilde{C}, \uHom(\tilde{h_y} \oT \tilde{D}, \hat{W}))
		} \nonumber
	\eeq
Because $W$ is $\Gamma$-local, $\uHom(\tilde{h_y} \oT \tilde{D}, \hat{W}) \rarrequ \uHom(\tilde{h_x} \oT \tilde{D}, \hat{W})$, which gives us the bottom horizontal equivalence above by Thm 17.7.7 of \cite{Hi}, hence a top horizontal one as well by commutativity, which shows $\uHom(\tilde{C}, \hat{W})$ is $\Gamma$-local. This completes the proof. 
\end{proof}

\begin{ideal2}
	A class of objects $\cC$ in a symmetric monoidal model category $\cM$ is said to be an \textbf{ideal class of objects} \index{object! ideal class} if for any $C \in \cC$, for any object $X$ of $\cM$, $X \otimes C$ is again in $\cC$.
\end{ideal2}
We need the following fact for having an internal right Bousfield localization:
\begin{iKcolobjideal}
In a right proper, cellular model category $\cM$, $K$ a class of objects in $\cM$, $\cK$ the class of i$K$-colocal equivalences (Def. \ref{colequ}), then i$\cK$-colocal objects (Def. \ref{colocal}) form an ideal class.\\
\end{iKcolobjideal}
\begin{proof}
Let $W$ be an i$\cK$-colocal object, $f: X \rarr Y$ an element of $\cK$. We have an equivalence of iLhfC's: $\uHom(\tilde{W}, \hat{X}) \rarrequ \uHom(\tilde{W}, \hat{Y})$. Now let $D \in \cM$.
\newline
We wish to show $\uHom(\widetilde{D \otimes W}, \hat{X}) \rarrequ \uHom(\widetilde{D \otimes W}, \hat{Y})$. It suffices to consider for all $Z$ cofibrant in $\cM$:
\beq
\xymatrix{
\Hom(Z, \uHom(\widetilde{D \otimes W}, \hat{X})) \ar@{=}[d] \ar@{.>}[r] &\Hom(Z, \uHom(\widetilde{D \otimes W}, \hat{Y})) \ar@{=}[d] \\
\Hom(Z, \uHom(\diag \tilde{D} \otimes \tilde{W}, \hat{X})) \ar[d]^{\cong} &\Hom(Z, \uHom(\diag \tilde{D} \otimes \tilde{W}, \hat{Y})) \ar[d]^{\cong} \\
\diag \Hom(Z \otimes \tilde{D}, \uHom(\tilde{W}, \hat{X})) \ar[r]_{\simeq} & \diag \Hom(Z \otimes \tilde{D}, \uHom(\tilde{W}, \hat{Y}))
} \nonumber
\eeq
having the bottom equivalence by definition of an i$\cK$-colocal object, it follows that the top horizontal map is an equivalence, hence by Thm 17.7.7 of \cite{Hi} we have our desired equivalence, that is $D \otimes W$ is i$\cK$-colocal.
\end{proof}

In the above proof, we used:
\begin{diagHom}
	\beq
\diag \Hom(Z \otimes \tilde{D}, \uHom(\tilde{W}, \hat{Y})) \cong \Hom(Z, \uHom(\diag (\tilde{D} \otimes \tilde{W}), \hat{Y}))  \nonumber
\eeq
\end{diagHom}
\begin{proof}
We check this componentwise:
\begin{align}
\diag \Hom(Z \otimes \tilde{D}, \uHom(\tilde{W}, \hat{Y}))_n &= \Hom((Z \otimes \tilde{D})_n, (\uHom(\tilde{W}, \hat{Y}))_n) \nonumber \\
&=\Hom(Z \otimes \tilde{D}_n, \uHom( \tilde{W}_n, \hat{Y})) \nonumber \\
&\cong \Hom(Z, \uHom(\tilde{D}_n \otimes \tilde{W}_n, \hat{Y})) \nonumber \\
&=\Hom(Z, \uHom(\diag (\tilde{D} \otimes \tilde{W})_n, \hat{Y})) \nonumber \\
&=\Hom(Z, \uHom(\diag (\tilde{D} \otimes \tilde{W}), \hat{Y})_n) \nonumber \\
&=\Hom(Z, \uHom(\diag (\tilde{D} \otimes \tilde{W}), \hat{Y}))_n \nonumber
\end{align}
\end{proof}

The first result for an internal localization is the following:
\begin{4.5.6} \label{4.5.6 mod}
If $\Gamma$ is a set of maps in the left proper, cellular model category $\DCop$, if $p:X \rarr Y$ has the right lifting property with respect to those inclusions of subcomplexes $i: C \rarr D$ that are i$\Gamma$-local equivalences and for which the size of $D$ is at most $\gamma$, the cardinal described in the previous section, then $p$ has the right lifting property with respect to all inclusions of subcomplexes that are i$\Gamma$-local equivalences.
\end{4.5.6}
\begin{proof}
	Suppose $p:X \rarr Y$ is a map that satisfies the conditions of the proposition. Let $i:C \rarr D$ be an inclusion of cell subcomplexes that is an i$\Gamma$-local equivalence and such that the size of $D$ is at most $\gamma$. It is also a $\Gamma$-local equivalence by Lemma \ref{lequ implies glequ}. Since $i$ is a $\Gamma$-local equivalence, $p$ has the right lifting property with respect to all inclusions of subcomplexes that are also $\Gamma$-local equivalences and for which the size of $D$ is at most $\gamma$. The original Prop. 4.5.6 of \cite{Hi} can then be used to conclude $p$  has the right lifting property with respect to all inclusions of subcomplexes that are $\Gamma$-local equivalences, which are also i$\Gamma$-local equivalences by Lemma \ref{lequ implies glequ}. This completes the proof.
\end{proof}
The second result is the following:
\begin{4.5.2} \label{4.5.2 mod}
If $\Gamma$ is a set of maps in the left proper, cellular model category $\DCop$, if $p:E \rarr B$ is a fibration with the right lifting property with respect to all inclusions of cell complexes that are i$\Gamma$-local equivalences, then it has the right lifting property with respect to all cofibrations that are i$\Gamma$-local equivalences.
\end{4.5.2}
\begin{proof}
	The reasoning is similar to the previous result; since i$\Gamma$-local equivalences are also $\Gamma$-local equivalences, the original Lemma 4.5.2 of \cite{Hi} applies, $p$ has the right lifting property with respect to all cofibrations that are $\Gamma$-local equivalences, which are also i$\Gamma$-local equivalences.
\end{proof}
\begin{gloc implies loc} \label{gloc implies loc}
The class of i$\Gamma$-local objects of $\DCop$ coincides with the class of $\Gamma$-local objects.
\end{gloc implies loc}
\begin{proof}
Let $W$ be an i$\Gamma$-local object in $\DCop$, let $A \rarr B$ be an object of $\Gamma$. Then $\uHom(\tilde{B}, \hat{W}) \rarr \uHom(\tilde{A},\hat{W})$ is an equivalence, which implies the equivalence $\Hom(\tilde{B}, \hat{W}) \rarr \Hom(\tilde{A}, \hat{W})$, which exactly means that $W$ is $\Gamma$-local as well. Conversely, let $W$ be $\Gamma$-local, let $A \rarr B$ in $\Gamma \subset \Gamma-loc.equ's$. Now those form an ideal class, so for any $C$, $C \otimes A \rarr C \otimes B$ is a $\Gamma$-local equivalence, hence we have an equivalence at the bottom of the commutative diagram below:
\beq
\xymatrix{
\Hom(\tilde{C}, \uHom(\tilde{B}, \hat{W})) \ar[r] \ar[d]_{\cong} & \Hom(\tilde{C}, \uHom(\tilde{A}, \hat{W})) \ar[d]^{\cong} \nonumber \\
\Hom(\tilde{C} \otimes \tilde{B}, \hat{W}) \ar@{.>}[r]_{\simeq} &\Hom(\tilde{C} \otimes \tilde{A}, \hat{W}) \nonumber
}
\eeq
and this for any $C$, so by Thm 17.7.7 of \cite{Hi}, it follows that $\uHom(\tilde{B}, \hat{W}) \rarrequ \uHom(\tilde{A}, \hat{W})$, that is $W$ is i$\Gamma$-local.
\end{proof}
\begin{lequ implies glequ} \label{lequ implies glequ}
The class of $\Gamma$-local equivalences coincides with the class of i$\Gamma$-local equivalences.
\end{lequ implies glequ}
\begin{proof}
Let $C \rarr D$ be a $\Gamma$-local equivalence. If $A$ is an object of $\DCop$, $A \otimes C \rarr A \otimes D$ is also a $\Gamma$-local equivalence since those form an ideal class. This means for all $W$ i$\Gamma$-local, in particular $\Gamma$-local by the preceding lemma:
\beq
\xymatrix{
\Hom(\widetilde{A \otimes D}, \hat{W}) \ar[d]_{\cong} \ar[r]^{\simeq} &\Hom(\widetilde{A \otimes C}, \hat{W}) \ar[d]^{\cong}\nonumber \\
\Hom(\tilde{A}, \uHom(\tilde{D}, \hat{W})) \ar@{.>}[r]_{\simeq} & \Hom(\tilde{A}, \uHom(\tilde{C},\hat{W})) \nonumber
}
\eeq
having the equivalence at the bottom of this diagram implies that we have an equivalence $\uHom(\tilde{D}, \hat{W}) \rarr \uHom(\tilde{C}, \hat{W})$ for all i$\Gamma$-local object $W$ by Thm 17.7.7 of \cite{Hi} so $C \rarr D$ is also an i$\Gamma$-local equivalence. Conversely, since equivalences of ihfC's imply equivalences of hfC's, it follows i$\Gamma$-local equivalences are also $\Gamma$-local equivalences, which completes the proof.
\end{proof}
\begin{4.5.1} \label{4.5.1 mod}
In the left proper, cellular model category $\DCop$, if $\Gamma$ is a set of maps, then there exists a set $J_{\Gamma}$ of inclusions of cell complexes such that the class of i$J_{\Gamma}$-cofibrations equals the class of cofibrations that are i$\Gamma$-local equivalences.
\end{4.5.1}
\begin{proof}
The proof is almost identical to that of \cite{Hi}, save that it is generalized. We let $J_{\Gamma}$ be a set of representatives of isomorphism classes  of inclusions of cell complexes that are i$\Gamma$-local equivalences of size at most $\gamma$. This cardinal is the one described in the previous section. The result follows, as in \cite{Hi} in the classical case, from Prop. \ref{4.5.6 mod} and Lemma \ref{4.5.2 mod}, as well as Cor. 10.5.22 of \cite{Hi}, which we use verbatim since it does not make use of a notion of equivalence.
\end{proof}

\section{Internal left Bousfield localization}
\begin{Glocal} \label{Glocal}
	For $\Gamma$ a class of maps in $\DCop$, $W \in \DCop$ is said to be \textbf{i$\Gamma$-local} \index{object! i$\Gamma$-local} if it is fibrant and for any $f:A \rarr B$ in $\Gamma$, the induced map of iRhfC's $f^*:\uHomDCop(\tilde{B},\hat{W}) \rarr \uHomDCop(\tilde{A}, \hat{W})$ is a weak equivalence in $\DCopDop$.
\end{Glocal}
\begin{Glocalequ} \label{Glocalequ}
	A map $g:X \rarr Y$ in $\DCop$ is said to be an \textbf{i$\Gamma$-local equivalence} \index{equivalence! i$\Gamma$-local} if for all i$\Gamma$-local object $W$, the induced map of iRhfC's $g^*: \uHomDCop(\tilde{Y}, \hat{W}) \rarr \uHomDCop(\tilde{X}, \hat{W})$ is a weak equivalence in $\DCopDop$.
\end{Glocalequ}
\begin{4.1.1} \label{Thm 4.1.1}
If $\Gamma$ is a class of maps in the proper, cellular model category $\DCop$, then if we define a class of equivalences on $\underline{\DCop}$ (the category underlying $\DCop$) as being i$\Gamma$-local equivalences, if we define cofibrations as being those of $\DCop$, and if define fibrations as being those maps having the right lifting property with respect to those maps that are cofibrations and i$\Gamma$-local equivalences, this defines a model structure on $\underline{\DCop}$ that we denote by $L_{\Gamma} \DCop = \DCopw$ and which we call an internal left Bousfield localization of $\DCop$ along $\Gamma$. Further, $\DCopw$ is a left proper, cellular model category.
\end{4.1.1}
\begin{proof}
	We adapt the classical proof, which uses Thm 11.3.1 of \cite{Hi}, which we use as is. The reader is referred to \cite{Hi} for its statement. By Prop. \ref{3.2.3 mod}, the class of i$\Gamma$-local equivalences satisfies the 2-3 property. This is one needed assumption of Thm 11.3.1. Another assumption about this class we need for Thm 11.3.1 to hold is that it be closed under retracts. This follows in the classical case from Prop. 3.2.4 of \cite{Hi}, which holds in the internal case as well. Consider now the set $J_{\Gamma}$ provided by Prop. \ref{4.5.1 mod}. Let $I$ be the set of generating cofibrations of $\DCop$. By definition, $I$ permits the small object argument, hence condition 1 of Thm 11.3.1 is satisfied. Since every element of $J_{\Gamma}$ has a cofibrant domain, small relative to the subcategory of cofibrations by Thm 12.4.3 of \cite{Hi} as argued in the same paper, hence in particular small relative to $J_{\Gamma}$, this latter satisfies condition 1 of Thm 11.3.1 as well. Indeed elements of $J_{\Gamma}$ are relative $I$-cell complexes, which are in $\Icof$ by Prop. 10.5.10 of \cite{Hi}, and this is the subcategory $cof$ of cofibrations, so $J_{\Gamma} \subset cof$. Condition 2 of Thm 11.3.1 is that $J_{\Gamma}\text{-}cof \subset \Icof \cap W$, with $W=i\Gamma-loc.equ's$, and this follows from Prop. \ref{4.5.1 mod}: $J_{\Gamma}\text{-}cof = cof \cap W = \Icof \cap W$. Condition 3 states $\Iinj \subset J_{\Gamma}\text{-}inj \cap W$. Prop. \ref{4.5.1 mod} implies $J_{\Gamma}\text{-}cof$ is a subcategory of $cof = \Icof$, hence $J_{\Gamma}\text{-}inf \supset \Iinj$. Finally $\Iinj = triv.fibr.$, in particular weak equivalences, which are i$\Gamma$-local equivalences, i.e. in $W$, so $\Iinj \subset J_{\Gamma}\text{-}inf \cap W$. The last condition of Thm 11.3.1, condition 4a, is satisfied by Prop. \ref{4.5.1 mod}, as argued in \cite{Hi}. This completes the proof that we have a model structure. To show it yields a left proper, cellular model category, we follow exactly the proof of \cite{Hi}. There is no manifest modification going to the general case, so we just refer the reader to the proof in \cite{Hi}.
\end{proof}

\section{Results needed in the proof of Thm 5.1.1}
We will prove a modified version of Thm 5.1.1 of \cite{Hi}, which necessitates that the model category we are localizing be right proper. For us that would be $\LGDCop$, which is left proper, but not necessarily right proper, even if $\DCop$ itself is right proper. For the sake of considering covers, we consider only those model categories $\LGDCop = \DCopw$ that are right proper as well. Since this may be a strong constraint, we will develop the notion of internal right Bousfield localization not from $\DCopw$, but from a generic right proper, cellular model category $\cM$.\\

\begin{3.2.4} \label{3.2.4 mod}
If $\cK$ is a class of maps in a model category $\cM$, the class of i$\cK$-colocal equivalences is closed under retracts.
\end{3.2.4}
\begin{proof}
Let $g:X \rarr Y$ be an i$\cK$-colocal equivalence, $f:V \rarr W$ a retract of $g$, $\hat{g}: \hat{X} \rarr \hat{Y}$ and $\hat{f}:\hat{V} \rarr \hat{W}$ fibrant approximations to $g$ and $f$ respectively such that $\hat{f}$ is a retract of $\hat{g}$. Let $C$ be an i$\cK$-colocal object. We want $\uHom(\tilde{C}, \hat{V})\rarr \uHom(\tilde{C}, \hat{W})$ to also be an equivalence. Let $X \in \cM$, with cosimplicial resolution $\tilde{X}$. Then consider:
\beq
\xymatrix{
\Hom(\widetilde{X \otimes C}, \hat{V}) \ar@{=}[d] \ar[r]^{\simeq} &\Hom(\widetilde{X \otimes C}, \hat{W}) \ar@{=}[d] \nonumber \\
\Hom(\diag \tilde{X} \otimes \tilde{C}, \hat{V}) \ar[d]_{\cong} & \Hom(\diag \tilde{X} \otimes \tilde{C}, \hat{W}) \ar[d]^{\cong} \\
\diag \Hom(\tilde{X}, \uHom(\tilde{C}, \hat{V})) \ar@{.>}[r] & \diag \Hom(\tilde{X}, \uHom(\tilde{C}, \hat{W})) }\nonumber
\eeq
i$\cK$-colocal objects form an ideal class, so $X \otimes C$ is again i$\cK$-colocal, hence $\cK$-colocal by Lemma \ref{gK-col implies K-col}. Also, $g$ i$\cK$-colocal equivalence is a $\cK$-colocal equivalence by Prop. \ref{KigK}, $f$ is a retract of $g$, so a $\cK$-colocal equivalence by the original Proposition of \cite{Hi}, hence the top map above is an equivalence. Hence we have an equivalence at the bottom of the above commutative diagram, and Thm 17.7.7 of \cite{Hi} allows us to conclude $\uHom(\tilde{C}, \hat{V}) \rarrequ \uHom(\tilde{C}, \hat{W})$, hence i$\cK$-colocal equivalences are closed under retracts.
\end{proof}

To prove the lifting argument in $\RKM$, we need to show that a trivial cofibration in $\RKM$ is also a trivial cofibration in $\cM$. The proof of this claim follows from the fact that weak equivalences are also i$\cK$-colocal equivalences, which we now prove:

\begin{equ implies K-colequ}
If $\cK$ is a class of maps in $\cM$, then weak equivalences in $\cM$ are also i$\cK$-colocal equivalences.
\end{equ implies K-colequ}
\begin{proof}
This follows readily from Thm \ref{17.7.7 mod} using diagonal objects.
\end{proof}

With this, we can now prove:
\begin{5.3.2} \label{5.3.2 mod}
Trivial cofibrations in $\RKM$ are also trivial cofibrations in $\cM$.
\end{5.3.2}
\begin{proof}
This is verbatim the proof of the original claim in \cite{Hi}, save that we use the internal results Prop. \ref{3.2.3 mod} and Prop. \ref{3.2.4 mod}.
\end{proof}

In the proof of the factorization axiom for model categories, we lift the factorization in the classical case to the internal one by invoking Lemma \ref{KigK} below, which itself follows from:

\begin{gK-col implies K-col} \label{gK-col implies K-col}
If $\cK$ is a class of maps in $\cM$, then i$\cK$-colocal objects are also $\cK$-colocal objects.
\end{gK-col implies K-col}
\begin{proof}
This readily follows from Prop. \ref{3.1.1}; internal equivalences imply classical equivalences.
\end{proof}

\begin{K-colequ implies gK-colequ} \label{KigK}
If $\cK$ is a class of maps in $\cM$, then $\cK$-colocal equivalences coincide with i$\cK$-colocal equivalences.
\end{K-colequ implies gK-colequ}
\begin{proof}
By Prop. \ref{3.1.1}, i$\cK$-colocal equivalences are also $\cK$-colocal equivalences. Conversely, if $g: X \rarr Y$ is a $\cK$-colocal equivalence, for any i$\cK$-colocal object $W$, for any $C \in \cM$, we can write:
\beq
\xymatrix{
\Hom(\widetilde{C \otimes W}, \hat{X}) \ar[d]_{\cong} \ar[r]^{\simeq} & \Hom(\widetilde{C \otimes W}, \hat{Y}) \ar[d]^{\cong} \nonumber \\
\diag \Hom(\tilde{C}, \uHom(\tilde{W}, \hat{X})) \ar@{.>}[r]_{\simeq} &\diag \Hom(\tilde{C}, \uHom(\tilde{W}, \hat{Y})) \nonumber
}
\eeq
since i$\cK$-colocal objects form an ideal class, $C \otimes W$ is also i$\cK$-colocal, hence $\cK$-colocal by the previous result, so we have the top horizontal equivalence above, hence the one at the bottom as well, and we conclude by using Thm 17.7.7 of \cite{Hi}.
\end{proof}

\section{Internal right Bousfield Localization}

\begin{colocal} \label{colocal}
	For $\cM$ a model category, $\cK$ a class of maps in $\cM$, an object $W$ of $\cM$ is \textbf{i$\cK$-colocal} \index{object! i$\cK$-colocal} if cofibrant and if for any $f:A \rarr B$ in $\cK$, we have an induced equivalence in $\cMDop$ of iLhfC's: $f_*: \uHomM(\tilde{W}, \hat{A}) \rarr \uHomM(\tilde{W}, \hat{B})$, $\tilde{W}$ a cosimplicial resolution of $W$ and $\hat{A}$ a fibrant approximation to $A$.
\end{colocal}
\begin{colequ} \label{colequ}
	A map $g: X \rarr Y$ in $\cM$ is an \textbf{i$\cK$-colocal equivalence} \index{equivalence! i$\cK$-colocal} if for any $\cK$-colocal object $W$, we have an equivalence in $\cMDop$ of iLhfC's: $g_*: \uHomM(\tilde{W}, \hat{X}) \rarr \uHomM(\tilde{W}, \hat{Y})$
\end{colequ}
The proof of Thm 5.1.1 of \cite{Hi} makes use of a further notion, which we generalize to our setting:
\begin{Kcolequ}
For $K$ a class of objects in $\cM$, a map $f:A \rarr B$ in $\cM$ is said to be an i$K$-colocal equivalence if for any object $k \in K$, we have $\uHomM(\tilde{k}, \hat{A}) \rarrequ \uHomM(\tilde{k}, \hat{B})$.
\end{Kcolequ}
We now state our theorem, much like Thm 5.1.1 of \cite{Hi}, which states the existence of an internal right Bousfield localization.
\begin{RBous}
If $\cM$ is a right proper, cellular model category, $K$ a set of objects of $\cM$, $\cK$ the class of i$K$-colocal equivalences, then the internal right Bousfield localization of $\cM$ is a model category structure on $\underline{\cM}$ with i$\cK$-colocal equivalences as weak equivalences, the same fibrations as $\cM$, and for cofibrations those maps that have the left lifting property with respect to those fibrations that are also i$\cK$-colocal equivalences.
\end{RBous}
\begin{proof}
The proof, as in \cite{Hi}, consists in making sure the axioms of a model category are met. We will only focus on those claims that are different from the proof of \cite{Hi}. The 2-3 property is satisfied because of the dual of Prop. \ref{3.2.3 mod}. The retract argument follows from Lemma 7.2.8 of \cite{Hi}, which we use as is, and Prop. \ref{3.2.4 mod}, following the same argument as in \cite{Hi}. The lifting argument involving a cofibration is immediate, since a cofibration in $\RKM$ has the left lifting property with respect to fibrations that are also i$\cK$-local equivalences, i.e. trivial fibrations in $\RKM$. For the lifting argument involving a trivial cofibration, this follows from Lemma \ref{5.3.2 mod}. Finally for the functorial factorization axiom, we show any map $g: X \rarr Y$ can be factored as $X \stackrel{p}{\rarr} W \stackrel{q}{\rarr} Y$, where $p$ is a trivial cofibration in $\RKM$, $p$ is a fibration in $\RKM$. This follows readily from the classical case by invoking Prop. \ref{KigK}. The same is true of the factorization where now $p$ would be a cofibration, and $p$ a trivial fibration. This completes the proof.
\end{proof}

\end{document}